\newif\ifdraft \drafttrue
\PassOptionsToPackage{dvipsnames}{xcolor}
\documentclass[onefignum,onetabnum,a4paper,final]{siamart190516}

\usepackage[utf8]{inputenc}
\usepackage[T1]{fontenc}
\usepackage{amssymb}
\usepackage{amsfonts}
\usepackage{amsmath}
\usepackage{graphicx}
\usepackage{tikz}
\usepackage{pgfplots}
\usepackage{subfig} 
\usepackage{textcomp}
\usepackage{mathtools}
\usepackage{geometry}
\usepackage{tgpagella}
\usepackage{comment}
\usepackage[shortlabels]{enumitem}

\usepackage{moreverb}
\usepackage{aligned-overset}
\usepackage[most]{tcolorbox}

\numberwithin{equation}{section}

\newcommand{\parperp}{\vDash}
\newcommand{\R}{\mathbb{R}}
\newcommand{\Rpos}{\mathbb{R}_{+}}
\newcommand{\N}{\mathbb{N}}

\newcommand{\A}{\mathcal{A}}

\newcommand{\tUbeta}{\widetilde{U}^{(\beta)}}

\newcommand{\tN}{\widetilde{N}}
\newcommand{\tNp}{\widetilde{N}^{(p)}}
\newcommand{\tNpmone}{\widetilde{N}^{(p-1)}}
\newcommand{\tNppar}{\widetilde{N}^{(\ppar)}}
\newcommand{\tNpperp}{\widetilde{N}^{(\pperp)}}

\newcommand{\tFbeta}{\widetilde{F}^{(\beta)}}
\newcommand{\tfbeta}{\widetilde{f}^{(\beta)}}

\newcommand{\Upar}{U_{\parallel}}

\newcommand{\Uparbeta}{U_{\parallel}^{(\beta)}}

\newcommand{\Fparbeta}{F_{\parallel}^{(\beta)}}

\newcommand{\fparbeta}{f_{\parallel}^{(\beta)}}

\newcommand{\ppar}{p_{\parallel}}

\newcommand{\pperp}{p_{\perp}}

\newcommand{\Cloctwo}{C_{\mathrm{loc},2}}
\newcommand{\Creg}{C_{\mathrm{reg}}}

\newcommand{\omegaeps}{\xi}

\newcommand{\omegav}{\omega_{\mathbf{v}}}

\newcommand{\omegave}{\omega_{\mathbf{ve}}}
\newcommand{\omegavf}{\omega_{\mathbf{vf}}}
\newcommand{\omegavef}{\omega_{\mathbf{vef}}}
\newcommand{\omegaef}{\omega_{\mathbf{ef}}}
\newcommand{\tomegaef}{\widetilde{\omega}_{\mathbf{ef}}}

\newcommand{\omegae}{\omega_{\mathbf{e}}}
\newcommand{\omegaf}{\omega_{\mathbf{f}}}

\newcommand{\eremk}{\hbox{}\hfill\rule{0.8ex}{0.8ex}}
\newcommand{\tbeta}{\widetilde{\beta}}
\newcommand{\tbetam}{|\widetilde{\beta}|}

\newcommand{\rv}{r_\mathbf{v}}
\newcommand{\re}{r_\mathbf{e}}
\newcommand{\rf}{r_\mathbf{f}}
\newcommand{\rhove}{\rho_\mathbf{ve}}
\newcommand{\rhoef}{\rho_\mathbf{ef}}
\DeclareMathOperator{\dist}{dist}
\DeclareMathOperator{\card}{card}
\DeclareMathOperator{\tr}{tr}

\newsiamremark{remark}{Remark}

\newcommand{\abs}[1]{\left\vert #1 \right\vert}
\newcommand{\skp}[1]{\left< #1 \right>}
\newcommand{\norm}[1]{\left\| #1 \right\|}
\newcommand{\ra}[0]{\rightarrow}

\renewcommand{\div}{\operatorname*{div}}
\newcommand{\supp}{\operatorname*{supp}}

\newcommand{\deta}{\partial_x^\eta}
\newcommand{\etam}{|\eta|}
\newcommand{\dbeta}{\partial_x^\beta}
\newcommand{\betam}{|\beta|}

\newcommand{\CacInt}{C_{\mathrm{int}}}

\newcommand{\Doneparf}{D_{\mathbf{f}_{1, \parallel}}}
\newcommand{\Dtwoparf}{D_{\mathbf{f}_{2, \parallel}}}
\newcommand{\Doneperpe}{D_{\mathbf{e}_{1, \perp}}}
\newcommand{\Dtwoperpe}{D_{\mathbf{e}_{2, \perp}}}
\newcommand{\Dparf}{D_{\mathbf{f}_{\parallel}}}
\newcommand{\Dparg}{D_{\gpar}}
\newcommand{\Dperpg}{D_{\gperp}}
\newcommand{\Dparperpg}{D_{\gparperp}}
\newcommand{\Dpars}{D_{\mathbf{s}_{\parallel}}}

\newcommand{\Dperpf}{D_{\mathbf{f}_\perp}}
\newcommand{\Dpare}{D_{\mathbf{e}_\parallel}}
\newcommand{\Dperpe}{D_{\mathbf{e}_\perp}}
\newcommand{\fipar}{\mathbf{f}_{i, \parallel}}
\newcommand{\fperp}{\mathbf{f}_{\perp}}

\newcommand{\eiperp}{\mathbf{e}_{i, \perp}}
\newcommand{\epar}{\mathbf{e}_{\parallel}}
\newcommand{\gpar}{\mathbf{g}_{\parallel}}
\newcommand{\gparperp}{\mathbf{g}_{\parperp}}
\newcommand{\gperp}{\mathbf{g}_{\perp}}
\newcommand{\bfe}{\mathbf{e}}
\newcommand{\bfv}{\mathbf{v}}
\newcommand{\bff}{\mathbf{f}}
\newcommand{\bfs}{\mathbf{s}}

\newcommand{\cV}{\mathcal{V}}
\newcommand{\cF}{\mathcal{F}}
\newcommand{\cE}{\mathcal{E}}
\newcommand{\cB}{\mathcal{B}}
\newcommand{\hcB}{\widehat{\mathcal{B}}}

\newcommand{\hc}{\hat{c}}
\newcommand{\hR}{\widehat{R}}
\newcommand{\hB}{\widehat{B}}
\newcommand{\hH}{\widehat{H}}

\newcommand{\betaperp}{\beta_{\perp}}
\newcommand{\tbetaperp}{\widetilde{\beta}_{\perp}}
\newcommand{\betapar}{\beta_{\parallel}}
\newcommand{\tbetapar}{\widetilde{\beta}_{\parallel}}

\newcommand{\betaparperp}{\beta_{\parperp}}
\newcommand{\tbetaparperp}{\widetilde{\beta}_{\parperp}}
\newcommand{\pparperp}{p_{\vDash}}
\newcommand{\boundY}{\mathcal{Y}}

\newcommand{\tR}{\widetilde{R}}
\newcommand{\tc}{\tilde{c}}
\newcommand{\hW}{\widehat{W}}
\newcommand{\BL}{\mathrm{BL}}
\newcommand{\BLzero}{\mathrm{BL}^1_{\alpha, 0, \Omega}}

\title{Weighted analytic regularity for the\\ integral fractional Laplacian in polyhedra}
\author{%
  Markus Faustmann%
  \thanks{%
    Institut für Analysis und Scientific Computing, TU Wien, A-1040 Wien, Austria%
  }%
  \and 
  Carlo Marcati%
  \thanks{Dipartimento di Matematica ``F. Casorati'', Universit{\`a} di Pavia,
    I-27100 Pavia, Italy%
      }%
  \and
  Jens Markus Melenk\footnotemark[1]%
  \and
  Christoph Schwab%
  \thanks{Seminar for Applied Mathematics, ETH Zurich, CH-8092 Zürich, Switzerland%
  \funding{The research of JMM is funded by the Austrian Science Fund (FWF) by
      the special research program {\it Taming complexity in PDE systems} (grant
      SFB F65). 
      }%
}
}
\begin{document}

\maketitle

\begin{abstract}
On polytopal domains in $\mathbb{R}^3$, we prove weighted analytic regularity of solutions to the Dirichlet problem for the integral fractional Laplacian with analytic right-hand side. Employing the Caffarelli-Silvestre extension allows to localize the problem and to decompose the regularity estimates into results on vertex, edge, face, vertex-edge, vertex-face, edge-face and vertex-edge-face neighborhoods of the boundary. Using tangential differentiability of the extended solutions, a bootstrapping argument based on Caccioppoli inequalities on dyadic decompositions of the neighborhoods provides weighted, analytic control of higher order solution derivatives.
\end{abstract}
\tableofcontents
\section{Introduction}
\label{sec:Intro}
On a bounded, polytopal domain $\Omega \subset \R^3$ with Lipschitz boundary $\partial\Omega$
comprising (the closure of) a finite union of plane, open polygons,
we consider the Dirichlet problem for the integral fractional Laplacian
\begin{equation}
  \label{eq:intro-eq}
  (-\Delta)^s u = f \text{ on }\Omega, \qquad u = 0 \;\text{ on }\; \R^d \setminus \overline{\Omega},
\end{equation}
with $0<s<1$, subject to a 
source term $f$ that is analytic in $\overline{\Omega}$.

As solutions to fractional PDEs typically exhibit a singular behaviour close to the whole boundary $\partial \Omega$ of the domain, the aim of this article is to capture this singular behaviour in Sobolev scales by introducing certain weight functions, which are powers of distances to vertices, edges or faces of the polytope and vanish on $\partial \Omega$. As such, we derive weighted analytic-type estimates for the variational solution $u$ in $\Omega$, which also extends the analysis of our previous work \cite{FMMS21} (on 2D polygons) to the 3D-case.

Our analysis will, as in the two-dimensional setting \cite{FMMS21},
be based on \emph{localization of \eqref{eq:intro-eq} 
through a local, divergence form, elliptic degenerate operator in dimension $4$}. 
Furthermore, the proof technique initiated in \cite{FMMS22_999,FMMS21} will also be used here:
we establish a base regularity shift of the variational solutions in $\Omega$ via the difference-quotient
technique due to Savar\'{e} \cite{Savare}, rather than by localization and Mellin-analysis as is customary
in the regularity analysis of elliptic PDEs in corner domains (see, e.g., \cite{MR162} and the references there).
This allows, largely building upon the general results in \cite{Savare,FMMS21}, 
for a more succinct proof
of a small regularity shift in fractional order, non-weighted Sobolev spaces.
Subsequently,
this regularity is inductively bootstrapped to arbitrary order of regularity via
local regularity estimates of Caccioppoli type on appropriately scaled balls 
in a Besicovitch covering of the domain. 
These local, analytic regularity estimates are subsequently assembled into 
a-priori bounds in weighted Sobolev spaces, with corner-, edge- and face-weight functions.

While structurally similar to our analysis of the two-dimensional case \cite{FMMS21}, the 
analysis in polyhedral domains brings additional technical difficulties: the coverings and local
regularity estimates exhibit a certain ``recursive by dimension of the singular set'' structure, 
reminiscent to the ``singular chains'' of M. Dauge in the analysis of the singularities of the 
Laplacean in polytopal domain in $\mathbb{R}^d$ for general dimension $d\geq 2$ in \cite{DaugeLNM}.
\subsection{Relation to previous work}
\label{sec:RelPrev}
As mentioned, 
the present analysis extends our work \cite{FMMS21} to polyhedral domains in $\mathbb{R}^3$, 
thereby being the first analytic regularity results for the integral fractional Laplacian 
in three space dimensions. 

Previous, recent work \cite{BN21} establishes essentially optimal finite regularity shifts in (non-weighted)
Besov spaces in general Lipschitz domains $\Omega\subset \mathbb{R}^d$ in arbitrary dimension $d\geq 2$, 
which are also applicable in the presently considered case. As compared with \cite{BN21}, we consider
a more restricted geometric setting of Lipschitz polyhedra $\Omega \subset \mathbb{R}^3$ with a finite
number of faces. As in \cite{BN21} and in the two-dimensional case \cite{FMMS21} we build the base
regularity shift on the techniques of Savare \cite{Savare}. To obtain the analytic regularity shifts,
however, we then employ coverings and local Caccioppoli-type estimates with inductive bootstrapping.
This is distinct from the analysis in \cite{BabGuo3dII,BabGuoI}, which is based on inductive bootstrapping
in finite-order, corner-weighted spaces of Kondrat'ev type. 
As in \cite{FMMS21}, we develop this regularity analysis for 
the four-dimensional, singular \emph{local} elliptic divergence-form PDE 
related to \eqref{eq:intro-eq} which was developed in \cite{caffarelli-stinga16}
and the references there.
%
\subsection{Impact on numerical methods}
\label{sec:ImpNum}
As is customary in the analysis of finite element methods (FEM) and 
in other recent works (e.g. \cite{borthagaray2022fractional} and the references there),
sharp regularity for variational solutions of \eqref{eq:intro-eq} will imply corresponding
convergence rate estimates of Galerkin approximations. 
Similar to the two-dimensional case, where analytic regularity of solutions to \eqref{eq:intro-eq} on bounded, polygonal domains $\Omega$, 
obtained in \cite{FMMS21}, implied exponential convergence bounds for corresponding
$hp$ FE Galerkin approximations in \cite{FMMS-hp},
the weighted analytic regularity estimates obtained in the present paper
form the foundation for proving \emph{exponential rates of convergence} 
of suitable families of $hp$-FEM in polyhedral domains $\Omega$ in a forthcoming work. 
\subsection{Structure of this text}
\label{sec:Struct}
Upon fixing some notation in the next subsection, we establish
the variational formulation of~\eqref{eq:intro-eq} in Section \ref{sec:Set}.
We also introduce the scales of boundary-, edge- and vertex-weighted Sobolev spaces 
in which we subsequently will establish analytic regularity shifts.
In Section~\ref{sec:StatMainRes}, we state our main regularity result, 
Theorem~\ref{thm:mainresult}. 
The proof of this theorem is developed in the remaining part of the paper.
Section~\ref{sec:RegExt} recapitulates a global regularity shift and 
localized interior regularity estimates for the extension problem,
which were proved in \cite{FMMS21}.
In Section~\ref{sec:LocTgReg},
local regularity for various tangential derivatives of the solution of the extension problem, 
in a vicinity of (smooth parts of) the boundary will be considered.
While the mathematical structure of the proofs is identical to the polygonal case in \cite{FMMS21},
the number of cases to be distinguished is larger than in the polygonal case: singular sets 
now have either dimension zero (vertices $\mathbf{v}$), one (edges $\mathbf{e}$) or two 
(faces $\mathbf{f}$). 
A somewhat larger number of combined cases (listed in Section~\ref{sec:PartOmega})
needs to be discussed item by item. 
These localized estimates are combined in Section~\ref{sec:WghHpPolygon} 
with covering arguments and scaling to establish the weighted analytic regularity.
Section~\ref{sec:Concl} gives a summary of our main results.
Appendix~\ref{sec:AppA} develops some elementary estimates related to fractional
norms, which are used in some of the arguments in the main text.
\subsection{Notation}
\label{sec:Nota}
The notation used here is largely consistent with our analysis in the polygonal setting
in \cite{FMMS21}. 
For open $\omega \subseteq \R^d$ and $t \in \N_0$, the spaces $H^t(\omega)$ 
are the classical Sobolev spaces of order $t$. For $t \in (0,1)$, 
fractional order Sobolev spaces are given in terms of the Aronstein-Slobodeckij seminorm 
$|\cdot|_{H^t(\omega)}$ and the full norm $\|\cdot\|_{H^t(\omega)}$ by 
\begin{align}
\label{eq:norm} 
|v|^2_{H^t(\omega)} = \int_{x \in \omega} \int_{z \in \omega} \frac{|v(x) - v(z)|^2}{\abs{x-z}^{d+2t}}
\,dz\,dx, 
\qquad \|v\|^2_{H^t(\omega)} = \|v\|^2_{L^2(\omega)} + |v|^2_{H^t(\omega)}, 
\end{align}
where we denote the Euclidean norm in $\R^d$ by $\abs{\;\cdot\;}$. 

For bounded Lipschitz domains $\Omega \subset \R^d$ and $t \in (0,1)$,  we additionally introduce  
\begin{align*}
\widetilde{H}^{t}(\Omega) \coloneqq \left\{u \in H^t(\R^d) \,: \, u\equiv 0 \; \text{on} \; \R^d \backslash \overline{\Omega} \right\}, 
\quad  \norm{v}_{\widetilde{H}^{t}(\Omega)}^2 \coloneqq \norm{v}_{H^t(\Omega)}^2 + \norm{v/r^t_{\partial\Omega}}_{L^2(\Omega)}^2,
\end{align*}
where $r_{\partial\Omega}(x)\coloneqq\operatorname{dist}(x,\partial\Omega)$ 
denotes the Euclidean distance of a point $x \in \Omega$ from the boundary $\partial \Omega$.
On $\widetilde{H}^t(\Omega)$ we have,  
by combining \cite[Lemma~{1.3.2.6}]{Grisvard} and \cite[Proposition~2.3]{acosta-borthagaray17}, the estimate 
\begin{equation}
\label{eq:Htildet-vs-HtRd}
\forall u \in \widetilde{H}^t(\Omega) \colon \quad 
\|u\|_{\widetilde{H}^t(\Omega)} \leq C |\tilde{u}|_{H^t(\R^d)} 
\end{equation}
for some $C > 0$ depending only on $t$ and $\Omega$. 
For $t \in (0,1)\backslash \{\frac 1 2\}$, 
the norms $\norm{\cdot}_{\widetilde{H}^{t}(\Omega)}$  
and $\norm{\cdot}_{H^{t}(\Omega)}$ are equivalent on $\widetilde{H}^{t}(\Omega)$, 
see, e.g., \cite[Sec.~{1.4.4}]{Grisvard}.
Furthermore, for $t > 0$, the space $H^{-t}(\Omega)$ 
denotes the dual space of $\widetilde{H}^t(\Omega)$, 
and we write $\skp{\cdot,\cdot}_{L^2(\Omega)}$ 
for the duality pairing that extends the $L^2(\Omega)$-inner product.
\bigskip

We denote by $\Rpos$ the positive real numbers.
For subsets $\omega \subset \R^d$, we will use the notation $\omega^+\coloneqq
\omega \times \Rpos$; in addition, for real $\boundY>0$, we write $ \omega^\boundY =
\omega\times (0, \boundY)$.
For any multi index $\beta = (\beta_1,\dots,\beta_d)\in \mathbb{N}^d_0$, we
denote $\dbeta = \partial^{\beta_1}_{x_1}\cdots \partial^{\beta_d}_{x_d}$ and
$\betam= \sum_{i=1}^d\beta_i$.
We adhere to the convention that empty sums are null, i.e., 
$\sum_{j=a}^b c_j =0$ when $b<a$; 
this even applies to the case where the terms $c_j$ may not be defined. 
We also follow the standard convention $0^0 = 1$.
\bigskip

We use the notation $\lesssim$ to abbreviate $\leq$ up to a generic constant $C>0$
that does not depend on critical parameters in our analysis.

\section{Setting and Statement of the Main Result}
\label{sec:Set}
There are several different ways to define the fractional Laplacian $(-\Delta)^s$ for $s \in (0,1)$. 
A classical definition on the full space ${\mathbb R}^d$ 
is in terms of the Fourier transformation ${\mathcal F}$, i.e., 
$({\mathcal F} (-\Delta)^s u)(\xi) = |\xi|^{2s} ({\mathcal F} u)(\xi)$. 
Alternative, equivalent definitions of $(-\Delta)^s$ are, e.g., 
via spectral, semi-group, or operator theory,~\cite{Kwasnicki} or via singular integrals.

In the following, we consider the integral fractional Laplacian defined 
pointwise for sufficiently smooth functions $u$ as the principal value integral  
\begin{align}\label{eq:fracLaplaceDef}
(-\Delta)^su(x) \coloneqq C(d,s) \; \text{P.V.} \int_{\R^d}\frac{u(x)-u(z)}{\abs{x-z}^{d+2s}} \, dz \quad \text{with} \quad
C(d,s)\coloneqq - 2^{2s}\frac{\Gamma(s+d/2)}{\pi^{d/2}\Gamma(-s)},
\end{align}
where $\Gamma(\cdot)$ denotes the Gamma function. 
We investigate the fractional differential equation 
\begin{subequations}\label{eq:modelproblem}
\begin{align}
 (-\Delta)^su &= f \qquad \text{in}\, \Omega, \\
 u &= 0 \quad \quad\, \text{in}\, \Omega^c\coloneqq\R^d \backslash \overline{\Omega},
\end{align}
\end{subequations}
where $s \in (0,1)$ and $f \in H^{-s}(\Omega)$ is a given right-hand side. 
Equation \eqref{eq:modelproblem} is understood in weak form: 
Find $u \in \widetilde{H}^s(\Omega)$ such that 
\begin{equation}
\label{eq:weakform}
a(u,v)\coloneqq \skp{(-\Delta)^s u,v}_{L^2(\R^d)} = \skp{f,v}_{L^2(\Omega)} 
\qquad \forall v \in \widetilde{H}^s(\Omega). 
\end{equation}
The bilinear form $a(\cdot,\cdot)$ has the 
alternative representation 
\begin{equation}
\label{eq:DefBil}
a(u,v) = 
 \frac{C(s)}{2} \int_{\R^d}\int_{\R^d} 
 \frac{(\tilde{u}(x)-\tilde{u}(z))(\tilde{v}(x) - \tilde{v}(z))}{\abs{x-z}^{d+2s}} \, dz \, dx 
\qquad \forall u,v \in \widetilde{H}^s(\Omega). 
\end{equation}
Observe that the domain of integration in the bilinear form $a(\cdot,\cdot)$ 
in \eqref{eq:DefBil} equals $(\Omega\times \R^d) \cup (\R^d \times \Omega)$.
Existence and uniqueness of a weak solution $u \in \widetilde{H}^s(\Omega)$ 
of \eqref{eq:weakform} follow from
the Lax--Milgram Lemma for any $f \in H^{-s}(\Omega)$, 
upon the observation
that the bilinear form 
$a(\cdot,\cdot): \widetilde{H}^s(\Omega)\times \widetilde{H}^s(\Omega)\to \R$ 
is continuous and coercive (observing that coercivity with respect to the
$\widetilde{H}^s(\Omega)$-norm follows from \eqref{eq:Htildet-vs-HtRd}).

The main result of this article
asserts that, provided the data $f$ is analytic in $\overline{\Omega}$,
the variational solution $u$ of \eqref{eq:modelproblem} admits
\emph{weighted analytic regularity} 
in a scale of boundary-, edge- and corner-weighted Sobolev spaces
in $\Omega$. 
To state the result, we introduce some notation.

In the following, we consider 
$\Omega \subset \R^3$ a bounded, Lipschitz polyhedron 
with boundary $\partial\Omega$ comprised of 
finitely many vertices, and straight edges and plane faces. 
In $\overline{\Omega}$, we denote by 
$\cV$ the set of vertices $\bfv$ and by 
$\cE$ the set of the (open) edges $\bfe$, and by
$\cF$ the set of the (open) faces $\bff$ of $\partial\Omega$. 
Evidently then, 
$\partial\Omega = \bigcup_{\cF} \bff \cup \bigcup_{\cE} \bfe \cup \bigcup_{\cV} \bfv$.

For 
$\bfv \in \cV$, 
$\bfe \in \cE$, and
$\bff \in \cF$,
we shall require the distance functions 
\begin{align*} 
  \rv(x)\coloneqq|x - \bfv|, 
  \qquad 
  \re(x)\coloneqq\inf_{y \in \bfe} |x - y|, 
  \qquad
  \rf(x)\coloneqq\inf_{y \in \bff} |x - y|,
  \quad 
  x\in \Omega,
\end{align*}
and corresponding (nondimensional) relative distances
\begin{align*}
  \rhove(x)\coloneqq \re(x)/\rv(x),
  \qquad 
  \rhoef(x)\coloneqq \rf(x)/\re(x).
\end{align*} 
\subsection{Partition of $\Omega$}
\label{sec:PartOmega}
For each vertex 
$\bfv \in \cV$, 
we denote by 
$\cE_{\bfv}\coloneqq \{\bfe \in \cE: \bfv \in \overline{\bfe}\}$
the set of all edges that meet at $\bfv$, and 
$\cF_{\bfv}\coloneqq \{\bff \in \cF: 
 \bff \cap \overline{\bfv} \ne \emptyset\}$
the set of all faces abutting at the vertex $\bfv$.
For any edge $\bfe \in \cE$, 
we define 
$\cV_{\bfe}\coloneqq \{\bfv \in \cV: 
 \bfv \in \overline{\bfe}\} 
 = 
 \partial {\bfe}$,
and 
$\cF_{\bfe}\coloneqq \{\bff \in \cF: \bff \cap \overline{\bfe} \ne \emptyset\}$ 
as the set of faces sharing the edge $\bfe$.

For any face $\bff\in\cF$,
$\cE_{\bff} \coloneqq \{\bfe \in \cE\,:\, \bfe \subset \partial{\bff} \}$ 
is the set of edges abutting the face $\bff$, and 
$\cV_{\bff} \coloneqq \{\bfv \in \cV\,:\, \bfv \subset \overline{\bff} \}$ is
the set of vertices contained in the face $\overline{\bff}$.

For fixed, sufficiently small $\omegaeps > 0$ and for 
$\bfv\in\cV$, $\bfe\in\cE$, $\bff\in\cF$,
we decompose $\Omega$ into various neighborhoods defined as 
  \begin{align*}
    \omegavef^{\omegaeps} & \coloneqq \{x \in \Omega\,:\, \rv(x) < \omegaeps \quad \wedge \quad \rhove(x) < \omegaeps \quad \wedge \quad \rhoef(x) < \omegaeps \},
    \\
    \omegave^{\omegaeps} &\coloneqq \{x \in \Omega\,:\, \rv(x) < \omegaeps \quad \wedge \quad \rhove(x) < \omegaeps \quad \wedge \quad \rhoef(x) \geq \omegaeps \quad \forall \bff \in \mathcal{F}_{\bfe} 
    \},
    \\
    \omegavf^{\omegaeps} &\coloneqq \{x \in \Omega\,:\, \rv(x) < \omegaeps \quad \wedge \quad \rhove(x) \geq \omegaeps \quad \wedge \quad \rhoef(x) < \omegaeps \quad \forall \bfe \in \mathcal{E}_{\bfv} \cap \mathcal{E}_{\bff} \},
    \\
   \omegav^{\omegaeps} & \coloneqq \{x \in \Omega\,:\, \rv(x) < \omegaeps \quad \wedge \quad \rhove(x) \geq \omegaeps \quad \wedge \quad \rhoef(x) \geq \omegaeps \quad \forall \bfe \in \mathcal{E}_{\bfv},\; \bff \in \mathcal{F}_{\bfv} \}, 
    \\
    \omegaef^{\omegaeps} &\coloneqq \{x \in \Omega\,:\, \rv(x) \geq \omegaeps \quad \wedge \quad \re(x) < \omegaeps^2 \quad \wedge \quad \rhoef(x) < \omegaeps \quad \forall \bfv \in \mathcal{V}_{\bfe} \},
    \\
    \omegae^{\omegaeps} &\coloneqq \{x \in \Omega\,:\, \rv(x) \geq \omegaeps \quad \wedge \quad \re(x) < \omegaeps^2 \quad \wedge \quad \rhoef(x) \geq \omegaeps \quad \forall \bfv \in \mathcal{V}_{\bfe},\; \bff\in \mathcal{F}_{\bfe} \},
\\
    \omegaf^{\omegaeps} &\coloneqq \{x \in \Omega\,:\, \rv(x) \geq \omegaeps \quad \wedge \quad \re(x) \geq \omegaeps^2 \quad \wedge \quad \rf(x) < \omegaeps^3 \quad \forall \bfv \in \mathcal{V}_{\bff}, \; \bfe \in \mathcal{E}_{\bff} \},
    \\
    \Omega_{\mathrm{int}}^{\omegaeps} 
    &\coloneqq \{x \in \Omega\,:\, \rv(x) \geq \omegaeps \quad \wedge \quad \re(x) \geq \omegaeps^2 \quad \wedge \quad 
                                   \rf(x) \geq \omegaeps^3 \quad \forall \bfv, \bfe, \bff \}.
  \end{align*}


\begin{figure}[ht]
\begin{center}
\begin{tikzpicture}[scale=1.2]
  \def\R{5}
  \def\A{60}
\def\Rin{1/3}

\draw (1/2*\R, 0) node [below]{$\mathbf{f}$} ;
\draw ({1/2*\R*cos(\A)-.1},{1/2*\R*sin(\A)}) node[above]{$\mathbf{f}'$};
\draw ({\R*cos(\A/2)-.3},{1/2*\R*sin(\A)}) node[above]{$\mathbf{f}''$};
\draw (0,0) node {\textbullet} node[left] {$\mathbf{e}$}; 
\draw (\R,0) node {\textbullet} node[right] {$\mathbf{e'}$}; 
\draw ({\R*cos(\A)}, {\R*sin(\A)}) node {\textbullet} node[left] {$\mathbf{e''}$}; 

  \draw[-] (0, 0) -- ({\R*cos(\A)}, {\R*sin(\A)});
  \draw[-] (0, 0) -- (\R, 0); 
  \draw[-] ({\R*cos(\A)}, {\R*sin(\A)}) -- (\R, 0); 
  
  \draw[dashed] (0, 0) -- ({\Rin*\R*cos(\A/3)},  {\Rin*\R*sin(\A/3)}); 
  \draw[dashed] (0, 0) -- ({\Rin*\R*cos(2*\A/3)},{\Rin*\R*sin(2*\A/3)}); 
  \draw[dashed] (\R, 0) -- ({\R-\Rin*\R*cos(\A/3)},  {\Rin*\R*sin(\A/3)}); 
  \draw[dashed] (\R, 0) -- ({\R-\Rin*\R*cos(2*\A/3)},{\Rin*\R*sin(2*\A/3)}); 
  \draw[dashed] ({\R*cos(\A)}, {\R*sin(\A)}) -- ({\R*cos(\A)+\Rin*\R*cos(-2*\A+\A/3)}, {\R*sin(\A)+\Rin*\R*sin(-2*\A+\A/3)}); 
  \draw[dashed] ({\R*cos(\A)}, {\R*sin(\A)}) -- ({\R*cos(\A)+\Rin*\R*cos(-2*\A+2*\A/3)}, {\R*sin(\A)+\Rin*\R*sin(-2*\A+2*\A/3)}); 
  
  \draw[dashed] ({\Rin*\R*cos(\A/3)}, {\Rin*\R*sin(\A/3)}) -- ({\R-\Rin*\R*cos(\A/3)},{\Rin*\R*sin(\A/3)}); 
  \draw[dashed] %
  ({\Rin*\R*cos(2*\A/3)}, {\Rin*\R*sin(2*\A/3)}) %
  --({\R*cos(\A)+\Rin*\R*cos(-2*\A+\A/3)}, {\R*sin(\A)+\Rin*\R*sin(-2*\A+\A/3)});
  \draw[dashed] %
  ({\R-\Rin*\R*cos(2*\A/3)}, {\Rin*\R*sin(2*\A/3)}) %
  --({\R*cos(\A)+\Rin*\R*cos(-2*\A+2*\A/3)}, {\R*sin(\A)+\Rin*\R*sin(-2*\A+2*\A/3)});

\draw [dashed,domain=0:\A] plot ({\Rin*\R*cos(\x)}, {\Rin*\R*sin(\x)});
\draw [dashed,domain=0:\A] plot ({\R-\Rin*\R*cos(\x)}, {\Rin*\R*sin(\x)});
\draw [dashed,domain=-2*\A:-\A] plot ({\R*cos(\A)+\Rin*\R*cos(\x)}, {\R*sin(\A)+\Rin*\R*sin(\x)});

\draw ({\R/2},{\R/3*sin(\A)}) node[below]{$\omega_{\mathbf{v}}$};
\draw ({\R/2},{2*\Rin/3*\R*sin(\A/3)}) node[below]{$\omega_{\mathbf{v}\mathbf{f}}$};
\draw (3/8*\R-0.3 ,{1/2*\R*sin(\A)-0.3}) node [above]{$\omega_{\mathbf{v}\mathbf{f'}}$} ;
\draw (\R-3/8*\R+0.4 ,{1/2*\R*sin(\A)-0.3}) node [above]{$\omega_{\mathbf{v}\mathbf{f''}}$} ;
 
\draw ({3/4*\Rin*\R*cos(\A/2)},{3/4*\Rin*\R*sin(\A/2)-0.2}) node [above]{$\omega_{\mathbf{v}\mathbf{e}}$} ;
\draw ({3/4*\Rin*\R*cos(\A/2)+0.1},{0.1}) node [above]{$\omega_{\mathbf{v}\mathbf{e}\mathbf{f}}$} ;
\draw ({1/2*\Rin*\R*cos(\A/6)+0.05},{3/4*\Rin*\R*sin(\A/2)+0.2}) node [above]{$\omega_{\mathbf{v}\mathbf{e}\mathbf{f'}}$} ;

\draw ({\R-3/4*\Rin*\R*cos(\A/2)},{3/4*\Rin*\R*sin(\A/2)-0.2}) node [above]{$\omega_{\mathbf{v}\mathbf{e'}}$} ;
\draw ({\R-3/4*\Rin*\R*cos(\A/2)-0.1},{0.1}) node [above]{$\omega_{\mathbf{v}\mathbf{e'}\mathbf{f}}$} ;
\draw ({\R-1/2*\Rin*\R*cos(\A/6)-0.05},{3/4*\Rin*\R*sin(\A/2)+0.25}) node [above]{$\omega_{\mathbf{v}\mathbf{e'}\mathbf{f'}}$} ;

\draw ({\R*cos(\A)+\Rin*\R*cos(-2*\A+2*\A/3)-0.25}, {\R*sin(\A)+\Rin*\R*sin(-2*\A+2*\A/3)}) node [above]{$\omega_{\mathbf{v}\mathbf{e''}}$} ;
\draw ({\R*cos(\A)+\Rin*\R*cos(-2*\A+1*\A/6)}, {\R*sin(\A)+\Rin*\R*sin(-2*\A+1*\A/6)}) node [above]{$\omega_{\mathbf{v}\mathbf{e''}\mathbf{f'}}$} ;
\draw ({\R*cos(\A)+\Rin*\R*cos(-2*\A+5*\A/6)}, {\R*sin(\A)+\Rin*\R*sin(-2*\A+5*\A/6)}) node [above]{$\omega_{\mathbf{v}\mathbf{e''}\mathbf{f''}}$} ;
\end{tikzpicture}
\begin{tikzpicture}[scale=1.2]
  \def\R{5}
  \def\A{60}
\def\Rin{3/7}

\draw (1/2*\R, 0) node [below]{$\mathbf{e''}$} ;
\draw ({3/6*\R*cos(\A/2)},  {3/6*\R*sin(\A/2)}) node [below]{$\mathbf{e}$} ;
\draw ({1/2*\R*cos(\A)-.1},{1/2*\R*sin(\A)}) node[above]{$\mathbf{e}'$};
\draw (0,0) node {\textbullet} node[left] {$\mathbf{v}$}; 

  \draw[-] (0, 0) -- ({\R*cos(\A)}, {\R*sin(\A)});
  \draw[-] (0, 0) -- (\R, 0); 
  \draw[] (0, 0) -- ({4/6*\R*cos(\A/2)},  {4/6*\R*sin(\A/2)}); 

\draw[](\R, 0) to[bend right] ({4/6*\R*cos(\A/2)},  {4/6*\R*sin(\A/2)});
\draw[]({4/6*\R*cos(\A/2)},  {4/6*\R*sin(\A/2)}) to[bend right] ({\R*cos(\A)}, {\R*sin(\A)}) ;

\draw [domain=0:\A] plot ({\R*cos(\x)}, {\R*sin(\x)});
\draw [dashed,domain=0:\A] plot ({\R*cos(\x)}, {\R*sin(\x)});
\draw [dashed,domain=0:\A] plot ({\R*cos(\x)}, {\R*sin(\x)});

\coordinate (A1) at ({4/6*\R*cos(\A/2)},  {4/6*\R*sin(\A/2)});
\coordinate (B1) at ({0.815*\R*cos(2*\A/6)},  {0.815*\R*sin(2*\A/6)});
\coordinate (C1) at  ({0.88*\R*cos(2*\A/5)},  {0.89*\R*sin(2*\A/5)});
\draw[blue,thick] (0, 0) -- (B1); 
\draw[dashed] (0, 0) -- (C1); 
\draw[blue,thick](A1) -- (C1); 
\draw[blue,thick] (C1) to[bend left] (B1);

\draw[cyan,dashed,thick] (C1) to[bend left] (B1);
\draw[magenta,dashed,thick](A1) -- (C1); 
\draw[cyan,dashed,thick] (0, 0) -- (B1); 

\coordinate (A1) at ({0.902*\R*cos(3*\A/15)},  {0.902*\R*sin(3*\A/15)});
\coordinate (B1) at ({0.96*\R*cos(4*\A/15)},  {0.96*\R*sin(4*\A/15)});
\coordinate (C1) at  ({0.88*\R*cos(2*\A/5)},  {0.89*\R*sin(2*\A/5)});
\draw[dashed] (0, 0) -- (B1); 
\draw[dashed] (0, 0) -- (C1); 
\draw[cyan,thick](0,0) -- (A1); 
\draw[cyan,thick] (C1) to[bend left=20] (B1);
\draw[cyan,thick] (B1) to[bend left=20] (A1);

\coordinate (A1) at ({4/6*\R*cos(\A/2)},  {4/6*\R*sin(\A/2)});
\coordinate (B1) at ({0.88*\R*cos(2*\A/5)},  {0.89*\R*sin(2*\A/5)});
\coordinate (C1) at  ({0.88*\R*cos(8*\A/15)},  {0.89*\R*sin(8*\A/15)});
\coordinate (D1) at ({0.805*\R*cos(2*\A/3)},  {0.79*\R*sin(2*\A/3)});

\draw[dashed] (0, 0) -- (C1); 
\draw[magenta,thick] (C1) -- (A1);
\draw[magenta,thick] (B1) to[bend right=20] (C1);
\draw[blue,thick] (C1)  to[bend right=20] (D1);
\draw[blue,thick] (0,0) -- (D1);

\draw[blue,dashed,thick](A1) -- (C1); 

\draw[magenta] ({0.86*\R*cos(2*\A/5)},  {1.0*\R*sin(2*\A/5)}) node[left] {$\omega_{\mathbf{ve}}$};
\draw[cyan] ({0.88*\R*cos(2*\A/5)},  {0.7*\R*sin(2*\A/5)}) node[right] {$\omega_{\mathbf{vf}}$};
\draw[blue] ({0.88*\R*cos(2*\A/5)},  {0.79*\R*sin(2*\A/5)}) node[left] {$\omega_{\mathbf{vef}}$};
\draw[blue] ({0.67*\R*cos(2*\A/5)},  {1.15*\R*sin(2*\A/5)}) node[right] {$\omega_{\mathbf{vef'}}$};
\end{tikzpicture}
\end{center}
\caption{\label{fig:vertex-notation} 
Notation near a vertex $\mathbf{v}$, left: top view of the vertex cone (the vertex $\mathbf{v}$ is behind, on a straight line to the barycenter of the triangle), right: side view of the vertex cone.}
\end{figure}
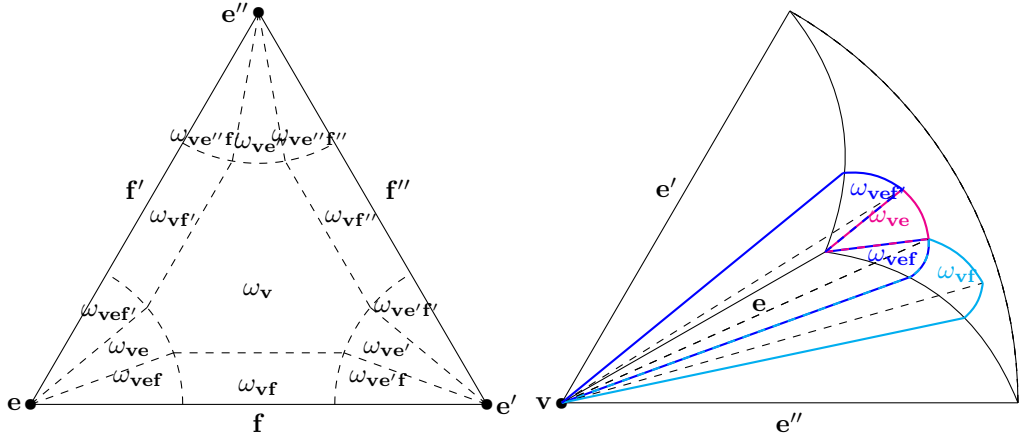

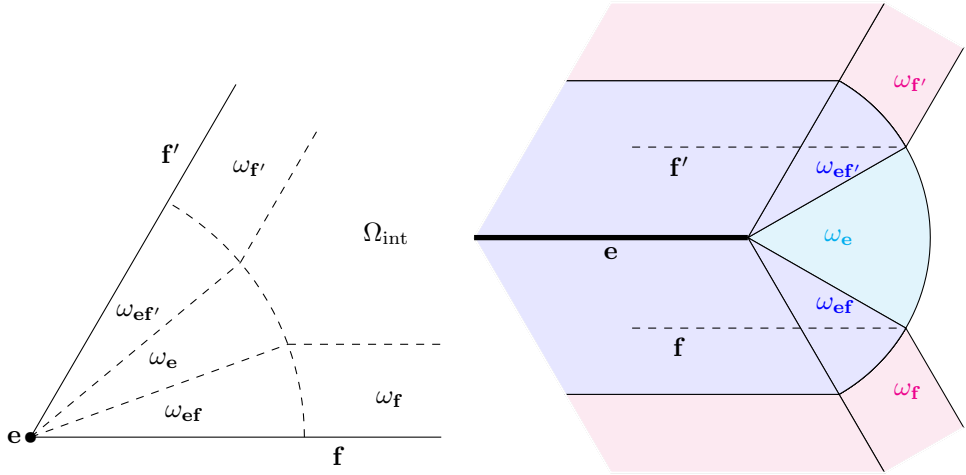
\begin{figure}[ht]
\begin{center}
\begin{tikzpicture}[scale=1.8]
  \def\R{3}
  \def\A{60}
\draw ({3/4*\R*cos(\A)-.1},{3/4*\R*sin(\A)}) node[above]{$\mathbf{f}'$};
\draw (0,0) node {\textbullet} node[left] {$\mathbf{e}$}; 
  \draw[-] (0, 0) -- ({\R*cos(\A)}, {\R*sin(\A)});
  \draw[-] (0, 0) -- (\R, 0); 
  \draw[dashed] (0, 0) -- ({2/3*\R*cos(\A/3)},  {2/3*\R*sin(\A/3)}); 
  \draw[dashed] (0, 0) -- ({2/3*\R*cos(2*\A/3)},{2/3*\R*sin(2*\A/3)}); 
  \draw[dashed] ({2/3*\R*cos(\A/3)}, {2/3*\R*sin(\A/3)}) -- (\R,{2/3*\R*sin(\A/3)}); 
  \draw[dashed] %
  ({2/3*\R*cos(2*\A/3)}, {2/3*\R*sin(2*\A/3)}) %
  -- ({2/3*\R*cos(2*\A/3) + \R*(1-2/3*cos(\A/3))*cos(\A)}, {2/3*\R*sin(2*\A/3) + \R*(1-2/3*cos(\A/3))*sin(\A)});
\draw (3/4*\R, 0) node [below]{$\mathbf{f}$} ;
\draw (3/8*\R ,0.05) node [above]{$\omega_{\mathbf{e}\mathbf{f}}$} ;
\draw (7/8*\R, 0.15) node [above]{$\omega_{\mathbf{f}}$} ;
\draw ({\R*cos(\A/2)}, {\R*sin(\A/2)}) node {$\Omega_{\mathrm{int}}$} ;
\draw ({3/8*\R*cos(\A/2)},{3/8*\R*sin(\A/2)-0.12}) node [above]{$\omega_{\mathbf{e}}$}; 
\draw ({3/8*\R*cos(3*\A/4)},{3/8*\R*sin(3*\A/4)}) node [above]{$\omega_{\mathbf{e}\mathbf{f}'}$} ;
\draw ({7/8*\R*cos(3*\A/4)-0.25},{7/8*\R*sin(3*\A/4)}) node [above]{$\omega_{\mathbf{f}'}$} ;
\draw [dashed,domain=0:\A] plot ({2/3*\R*cos(\x)}, {2/3*\R*sin(\x)});
\end{tikzpicture}
\hspace{3mm}
\begin{tikzpicture}[scale=1.2]
  \def\R{3}
  \def\A{60}
\draw[fill=magenta!10,domain=\A/2:\A] plot ({2/3*\R*cos(\x)}, {2/3*\R*sin(\x)}) -- ({\R*cos(\A)}, {\R*sin(\A)}) -- ({2/3*\R*cos(\A/2) + \R*(1-2/3*cos(\A/2))*cos(\A)}, {2/3*\R*sin(\A/2) + \R*(1-2/3*cos(\A/2))*sin(\A)}) -- cycle;
\draw[fill=magenta!10,domain=\A/2:\A] plot ({2/3*\R*cos(\x)}, {2/3*\R*sin(\x)}) -- (0,0) -- ({2/3*\R*cos(\A/2)}, {2/3*\R*sin(\A/2)});
\draw[fill=magenta!10]  ({2/3*\R*cos(\A)}, {2/3*\R*sin(\A)}) -- ({\R*cos(\A)}, {\R*sin(\A)}) -- ({\R*cos(\A)-\R}, {\R*sin(\A)}) -- ({2/3*\R*cos(\A)-\R}, {2/3*\R*sin(\A)}) -- cycle;

\draw[fill=magenta!10,domain=\A/2:\A] plot ({2/3*\R*cos(\x)}, {-2/3*\R*sin(\x)}) -- ({\R*cos(\A)}, {-\R*sin(\A)}) -- ({2/3*\R*cos(\A/2) + \R*(1-2/3*cos(\A/2))*cos(\A)}, {-2/3*\R*sin(\A/2) - \R*(1-2/3*cos(\A/2))*sin(\A)}) -- cycle;
\draw[fill=magenta!10,domain=\A/2:\A] plot ({2/3*\R*cos(\x)}, {-2/3*\R*sin(\x)}) -- (0,0) -- ({2/3*\R*cos(\A/2)}, {-2/3*\R*sin(\A/2)});
\draw[fill=magenta!10]  ({2/3*\R*cos(\A)}, {-2/3*\R*sin(\A)}) -- ({\R*cos(\A)}, {-\R*sin(\A)}) -- ({\R*cos(\A)-\R}, {-\R*sin(\A)}) -- ({2/3*\R*cos(\A)-\R}, {-2/3*\R*sin(\A)}) -- cycle;

\draw[fill=cyan!10,domain=-\A/2:\A/2] plot ({2/3*\R*cos(\x)}, {2/3*\R*sin(\x)}) -- (0,0) -- ({2/3*\R*cos(-\A/2)}, {2/3*\R*sin(-\A/2)});

\draw[fill=blue!10,domain=\A/2:\A] plot ({2/3*\R*cos(\x)}, {2/3*\R*sin(\x)}) -- ({2/3*\R*cos(\A)-\R}, {2/3*\R*sin(\A)}) -- ({2/3*\R*cos(\A/2)-\R}, {2/3*\R*sin(\A/2)}) -- ({2/3*\R*cos(\A/2)}, {2/3*\R*sin(\A/2)});
\draw[fill=blue!10,domain=\A/2:\A] plot ({2/3*\R*cos(\x)}, {2/3*\R*sin(\x)}) -- (0,0) -- ({2/3*\R*cos(\A/2)}, {2/3*\R*sin(\A/2)});
\draw[fill=blue!10]  ({2/3*\R*cos(\A)}, {2/3*\R*sin(\A)}) -- (0,0) -- (-\R,0) -- ({2/3*\R*cos(\A)-\R}, {2/3*\R*sin(\A)}) -- cycle;

\draw[fill=blue!10,domain=\A/2:\A] plot ({2/3*\R*cos(\x)}, {-2/3*\R*sin(\x)}) -- ({2/3*\R*cos(\A)-\R}, {-2/3*\R*sin(\A)}) -- ({2/3*\R*cos(\A/2)-\R}, {-2/3*\R*sin(\A/2)}) -- ({2/3*\R*cos(\A/2)}, {-2/3*\R*sin(\A/2)});
\draw[fill=blue!10,domain=\A/2:\A] plot ({2/3*\R*cos(\x)}, {-2/3*\R*sin(\x)}) -- (0,0) -- ({2/3*\R*cos(\A/2)}, {-2/3*\R*sin(\A/2)});
\draw[fill=blue!10]  ({2/3*\R*cos(\A)}, {-2/3*\R*sin(\A)}) -- (0,0) -- (-\R,0) -- ({2/3*\R*cos(\A)-\R}, {-2/3*\R*sin(\A)}) -- cycle;

\draw[dashed] ({2/3*\R*cos(\A/2)-\R}, {2/3*\R*sin(\A/2)}) -- ({2/3*\R*cos(\A/2)}, {2/3*\R*sin(\A/2)});
\draw[dashed] ({2/3*\R*cos(-\A/2)-\R}, {2/3*\R*sin(-\A/2)}) -- ({2/3*\R*cos(-\A/2)}, {2/3*\R*sin(-\A/2)});

\draw[white,thick] ({2/3*\R*cos(\A/2) + \R*(1-2/3*cos(\A/2))*cos(\A)}, {2/3*\R*sin(\A/2) + \R*(1-2/3*cos(\A/2))*sin(\A)}) -- ({\R*cos(\A)}, {\R*sin(\A)}) -- ({\R*cos(\A)-\R}, {\R*sin(\A)})--({-\R}, {0}) -- ({\R*cos(\A)-\R}, {-\R*sin(\A)}) -- ({\R*cos(\A)}, {-\R*sin(\A)})--({2/3*\R*cos(\A/2) + \R*(1-2/3*cos(\A/2))*cos(\A)}, {-2/3*\R*sin(\A/2) - \R*(1-2/3*cos(\A/2))*sin(\A)});

\draw[] (-\R/2,0) node[below] {$\mathbf{e}$};
\draw[] (-\R/4,-\R/3) node[below] {$\mathbf{f}$};
\draw[] (-\R/4,\R/3) node[below] {$\mathbf{f}'$};
\draw[line width=2] (-\R,0) -- (0,0);

\draw[cyan] (\R/3,0) node[] {$\omega_{\mathbf{e}}$};
\draw[blue] (\R/4-0.1,\R/4) node[right] {$\omega_{\mathbf{ef'}}$};
\draw[blue] (\R/4-0.1,-\R/4) node[right] {$\omega_{\mathbf{ef}}$};

\draw[magenta] (\R/2,\R/2+0.2) node[right] {$\omega_{\mathbf{f'}}$};
\draw[magenta] (\R/2,-\R/2-0.2) node[right] {$\omega_{\mathbf{f}}$};
\end{tikzpicture}
\end{center}
\caption{\label{fig:edge-notation}
Notation near an edge $\mathbf{e}$ with two faces $\mathbf{f},\mathbf{f}'$ 
meeting at the edge and no vertex close by, left: front view (edge collapses to point), right: side view.}
\end{figure}

Figure~\ref{fig:vertex-notation} 
illustrates the neighborhoods near a vertex and Figure~\ref{fig:edge-notation} 
shows the neighborhoods close to an edge but away from a vertex.
We drop the superscript $\omegaeps$ unless strictly necessary.

\noindent
{\bf Decompositions}:
We decompose the Lipschitz polyhedron $\Omega$ into (possibly overlapping) sectorial neighborhoods of vertices $\mathbf{v}$, 
which are unions of 
vertex, vertex-edge, vertex-face, and vertex-edge-face neighborhoods 
(as depicted in Figure~\ref{fig:vertex-notation}),
wedge-shaped neighborhoods of edges $\mathbf{e}$ 
(that are bounded away from a vertex, but are unions of edge- and edge-face neighborhoods 
as depicted in Figure~\ref{fig:edge-notation}), 
neighborhoods of faces $\mathbf{f}$, 
and an interior $\Omega_{\rm int}$, 
i.e., 
\begin{align}\label{eq:Nghbrhoods}
 \Omega = \Omega_{\rm int} \cup \bigcup_{\mathbf{v} \in \mathcal{V}}\left( \omega_{\mathbf{v}} 
           \cup \bigcup_{\mathbf{e} \in \mathcal{E}_{\mathbf{v}}, \; \mathbf{f} \in \mathcal{F}_{\mathbf{v}}} 
           \omega_{\mathbf{ve}} \cup \omega_{\mathbf{vf}} \cup \omegavef \right) 
           \cup \bigcup_{\mathbf{e} \in \mathcal{E}} \left( \omega_{\mathbf{e}} 
           \cup \bigcup_{\mathbf{f} \in \mathcal{F}_{\mathbf{e}}} \omega_{\mathbf{ef}} \right)
           \cup \bigcup_{\mathbf{f} \in \mathcal{F}} \omega_{\mathbf{f}}
\;.
\end{align}
Each sectoral and edge neighborhood may have a different value $\omegaeps$, 
but we assume that each $\omega_{\bullet}$ abuts at 
most one vertex, one edge, or one face of $\partial\Omega$. 
Since only finitely many distinct types of neighborhoods are needed to decompose the polygon, 
the interior $\Omega_{\rm int}\subset\Omega $ has a positive distance from the boundary.  
\subsection{Coordinates}
\label{sec:Coord}
To state the main result, and throughout the ensuing proof of analytic
estimates, we require coordinates
tangential resp. perpendicular to edges $\bfe$ and faces $\bff$ in the local neighborhoods.
\begin{definition} \label{def:Coord}[Co-ordinates and directional derivatives in neighborhoods of singular sets] 
\newline
\begin{enumerate}
\item
In {\bf face or vertex-face neighborhoods} $\omegaf$, $\omegavf$,
we let $\mathbf{f}_{i,\parallel}$, $i=1,2$ 
and $\mathbf{f}_{\perp}$ 
be unit vectors 
such that $\mathbf{f}_{i,\parallel}$ are mutually orthogonal and 
span the tangential plane to $\mathbf{f}$, 
and $\mathbf{f}_{\perp}$ is normal to $\mathbf{f} \in \mathcal{F}$. 
We assume that $\mathbf{f}_{\perp}$ and $\mathbf{f}_{i,\parallel}$ 
are right-oriented.
\item
In 
{\bf edge or vertex-edge neighborhoods} $\omegae$, $\omegave$, 
we let
$\mathbf{e}_{\parallel}$ and $\mathbf{e}_{1,\perp}$, $\mathbf{e}_{2,\perp}$ 
be unit vectors
such that $\mathbf{e}_{\parallel}$ is tangential to $\mathbf{e}$ 
and 
$\mathbf{e}_{i,\perp}$ are mutually orthogonal and span the plane transversal to $\mathbf{e}$.

\item
In {\bf edge-face or vertex-edge-face neighborhoods} $\omegaef$, $\omegavef$,
we choose three linearly independent, right-oriented 
unit vectors $\{\gpar,\gparperp,\gperp\}$ 
satisfying
\begin{itemize}
\item $\gpar$ is parallel to $\bfe$ and $\bff$;
\item $\gparperp$ is perpendicular to $\bfe$ and parallel to $\bff$;
\item $\gperp$ is perpendicular to $\bfe$ and $\bff$.
\end{itemize}
\end{enumerate}

For $\bfs\in \{\eiperp, \epar, \fperp, \fipar, \gpar, \gparperp, \gperp \}$ 
we denote first order derivatives as
$D_{\bfs} v \coloneqq \bfs \cdot \nabla_x v$.
For higher order derivatives, we set
\begin{equation*}
  D_{\bfs}^k v \coloneqq D_{\bfs}(D_{\bfs}^{k-1} v) \quad \text{for } k>1.
\end{equation*}
Finally, for $\beta = (\beta_1, \beta_2)\in \N^2_0$, we write
\begin{equation*}
  \Dperpe^{\beta} = \Doneperpe^{\beta_1}\Dtwoperpe^{\beta_2}, \qquad 
  \Dparf^{\beta} = \Doneparf^{\beta_1}\Dtwoparf^{\beta_2}.
\end{equation*}
\end{definition}
The coordinates introduced above can be written in a unified way. 
The following definition formalizes the notation 
used to write the statement of our main result and the proofs 
in a compact form. 
\begin{definition}\label{def:gs}
  Let $\omega \subset \Omega$ be any connected set abutting at most one vertex $\bfv$,
  one edge $\bfe$, and one face $\bff$ of $\partial \Omega$.
  We take $(\gperp, \gparperp, \gpar)$ 
   to be linearly independent unit vectors in $\R^3$ that additionally satisfy
    \begin{itemize}
    \item $\gperp$ is perpendicular to $\bff$ if $\bff \cap \partial \omega\neq \emptyset$
    and perpendicular to $\bfe$ if $\bfe \cap \partial \omega\neq \emptyset$;
  \item $\gparperp$ is parallel to $\bff$ if $\bff \cap \partial \omega\neq \emptyset$
    and perpendicular to $\bfe$ if $\bfe \cap \partial \omega\neq \emptyset$;
  \item $\gpar$ is parallel to $\bff$ if $\bff \cap \partial \omega\neq \emptyset$
    and parallel to $\bfe$ if $\bfe \cap \partial \omega\neq \emptyset$.
    \end{itemize}
  With these vectors and for  $\beta = (\betaperp, \betaparperp, \betapar)\in\N^3_0$, we introduce the derivative 
 \begin{align*}
 D_{(\gperp, \gparperp, \gpar)}^{\beta} =  D_{\gperp}^{\betaperp}D_{\gparperp}^{\betaparperp}D_{\gpar}^{\betapar}.
 \end{align*} 
\end{definition}
\subsection{Statement of the main result}
\label{sec:StatMainRes}
The following statement is the main result of this work.
It provides weighted analytic regularity in all neighborhoods used to decompose $\Omega$.
\begin{theorem}\label{thm:mainresult}
Let $\Omega \subset \R^3$ be a bounded, open Lipschitz polyhedron
whose boundary $\partial\Omega$ comprises
finitely many vertices, straight edges and plane faces.

Let the data $f \in C^{\infty}(\overline{\Omega})$ satisfy
with a constant $\gamma_f>0$
  \begin{equation}
    \label{eq:analyticdata}
\forall j \in \N_0\colon \quad 
    \sum_{\betam = j} \|\dbeta f\|_{L^2(\Omega)} \leq \gamma_f^{j+1}j^j. 
  \end{equation}
Let $u$ be the weak solution of \eqref{eq:modelproblem}.

Then, 
there exists $\gamma>0$ depending only on $\gamma_f$, $s$, 
and $\Omega$ such that for all $t<1/2$, 
    there exists $C_t>0$ such that
    for all $\beta = (\betaperp, \betaparperp, \betapar)\in\N^3_0$ 
    and all $\omega\subset\Omega$ as in Definition \ref{def:gs}, 
    it holds that 
  \begin{equation*}
    \| r_{\partial \Omega}^{-t-s} r_{\bfv}^{\betapar} r_{\bfe}^{\betaparperp} r_{\bff}^{\betaperp} 
            D_{(\gperp, \gparperp, \gpar)}^{\beta} u \|_{L^2(\omega)} \leq C_t \gamma^{\betam}\betam^{\betam}
  \end{equation*}
  with $\bfv$, $\bfe$, $\bff$ being the closest vertex, edge, face to $\omega$.
\end{theorem}
The rest of this paper will develop the proof of these bounds.
\section{The Caffarelli-Silvestre extension}
\label{sec:CSExt}
Key to the present regularity analysis is a localization 
of the fractional Laplacian provided by the so-called 
\emph{Caffarelli-Silvestre extension}, \cite{CafSil07}:
the nonlocal operator $(-\Delta)^s$ 
can be realized via a Dirichlet-to-Neumann map 
of a degenerate, \emph{local} elliptic PDE on a half space in $\R^{d+1}$. 
Here, we shall be mainly interested in $d=3$. 

\subsection{Weighted spaces for the Caffarelli-Silvestre extension}
\label{sec:WghtSpcCS}
We recapitulate from \cite{FMMS21}
certain weighted function spaces which will be used in the sequel.
We distinguish the last component of points in $\R^{d+1}$ 
with the notation $(x,y)$ where
$x  = (x_1,\ldots,x_d)\in \R^d$, $y \in \R$ 
and we set
\begin{equation}
\label{eq:alpha}
\alpha \coloneqq 1-2s. 
\end{equation}
For open sets $D \subset \R^d \times \Rpos$, 
the weighted $L^2$-norm $\| \cdot \|_{L^2_\alpha(D)}$ is defined 
via
\begin{equation}\label{eq:wgtL2}
 \norm{U}_{L^2_\alpha(D)}^2 \coloneqq \int_{(x,y) \in D} y^{\alpha} \abs{U(x,y)}^2 dx \, dy.
\end{equation}
For the variational formulation of the CS extension,
we require the space $L^2_\alpha(D)$ of functions on $D$ that are
square (Lebesgue-)integrable with respect to the weight $y^\alpha$. 
With the weighted space 
$H^1_\alpha(D):= \{U \in L^2_\alpha(D)\,\colon\, \nabla U \in L^2_\alpha(D)\}$
we introduce the Beppo-Levi space \cite{deny-lions55}
\begin{equation}\label{eq:BL1a}
{\BL}^1_{\alpha}\coloneqq \{U \in L^2_{loc}(\R^{d} \times \Rpos)\,:\, 
\nabla U \in L^2_\alpha(\R^d \times \Rpos)\}.
\end{equation}
Elements $U \in {\BL}^1_\alpha$
admit a trace at $y=0$,
which we denote as 
$\operatorname{tr} U$. 
It holds that (e.g., \cite[Lem.~3.8]{KarMel19})
$\operatorname{tr} U \in H^s_{loc}(\R^d)$. 
Also, for
$\operatorname{supp} \operatorname{tr} U \subset \overline{\Omega}$ 
for a bounded Lipschitz domain $\Omega$, 
$\operatorname{tr} U \in \widetilde H^{s}(\Omega)$ 
and 
\begin{align}
\label{eq:L3.8-KarMel19}
 \|\operatorname{tr} U\|_{\widetilde{H}^s(\Omega)} \stackrel{\eqref{eq:Htildet-vs-HtRd}}{\lesssim} |\operatorname{tr} U|_{H^s(\R^d)} 
\stackrel{\text{\cite[Lem.~{3.8}]{KarMel19}}}{\lesssim} \norm{\nabla U}_{L^2_\alpha(\R^d \times \Rpos)}
\end{align}
with implied constant depending on $s$ and $\Omega$.
\subsection{Statement of the Caffarelli-Silvestre extension}
\label{sec:CSExt-sub}
Given $u \in \widetilde{H}^s(\Omega)$, 
let $U = U(x,y)$ denote the (unique in ${\BL}^1_{\alpha}$, see \cite{FMMS21}) 
minimum norm extension of $u$ to $\R^d \times \Rpos$, i.e., 
$$
U = \operatorname{arg min} 
\{\|\nabla U\|^2_{L^2_\alpha(\R^d \times \Rpos)}\,|\, 
   U \in {\BL}^1_\alpha,\, 
   \operatorname{tr} U = u \;\mbox{in}\; H^s(\R^d) \}.
$$ 

The Euler-Lagrange equations corresponding to this extension problem read
\begin{subequations}\label{eq:extension}
\begin{align}
\label{eq:extension-a}
 \div (y^\alpha \nabla U) &= 0  \;\quad\quad\text{in} \; \R^d \times (0,\infty), \\ 
\label{eq:extension-b}
 U(\cdot,0) & = u  \,\quad\quad\text{in} \; \R^d.
\end{align}
\end{subequations}
Henceforth, when referring to solutions of \eqref{eq:extension}, 
we will additionally understand that $U \in {\BL}^1_\alpha$. 

The relevance of \eqref{eq:extension} is due to the fact that
the fractional Laplacian applied to $u\in \widetilde{H}^s(\Omega)$ 
can be recovered as distributional normal trace of 
the extension problem \cite[Section 3]{CafSil07}, \cite{caffarelli-stinga16}:
\begin{align}\label{eq:DtNoperator}
(-\Delta)^s u =  -d_s \lim_{y\ra 0^+} y^\alpha \partial_y U(x,y),
\qquad 
d_s = 2^{2s-1}\Gamma(s)/\Gamma(1-s). 
\end{align}
\subsection{Variational Formulation of the CS Extension}
\label{sec:VarCSExt}
Fix $\boundY > 0$. 
Given $F \in L^2_{-\alpha}(\R^d \times (0,\boundY))$ and $f \in H^{-s}(\Omega)$,
consider the problem to
find the minimizer $U = U(x,y)$ with $x \in \R^d$ and $y \in \Rpos$ of 
\begin{align} \label{eq:minimization}
\mbox{ minimize ${\mathcal F}$ on 
${\BL}^1_{\alpha,0,\Omega}
\coloneqq  
\{U \in {\BL}^1_{\alpha}\,:\, \operatorname{tr} U =0 \; \mbox{ on $\Omega^c$}\},$ }
\end{align}
where 
\begin{align}
\label{eq:calFdef}
\mathcal{F}(U)& \coloneqq  \frac{1}{2} b(U,U) - \int_{\R^d \times (0,\boundY)} F U \, dx\,dy - \int_{\Omega} f \operatorname{tr} U \,dx, & 
b(U,V) & := \int_{\R^d \times \Rpos} y^\alpha \nabla U \cdot \nabla V\,dx\,dy. 
\end{align}
In virtue of a Poincar\'e inequality (\cite[Lemma 3.1]{FMMS21}),
the map ${\BL}^1_{\alpha,0,\Omega} \ni U \mapsto \|\nabla U\|_{L^2_\alpha(\R^d \times \R_+)}$ is a norm.
The  space ${\BL}^1_{\alpha,0,\Omega}$ 
endowed with this norm is a Hilbert space
with corresponding inner-product given by the bilinear form $b(\cdot,\cdot)$ in \eqref{eq:calFdef}.
Hence,
for every $\boundY\in (0,\infty)$, there is $C_{\boundY,\alpha} > 0$ such that 
\begin{equation}
\label{eq:lemma:properties-of-H1alpha}
\forall U \in {\BL}^1_{\alpha,0,\Omega}\colon \quad 
\|U\|_{L^2_{\alpha}(\R^d \times (0,\boundY))} \leq C_{\boundY,\alpha} \|\nabla U\|_{L^2_\alpha(\R^d\times \R_+)}.  
\end{equation}
Details of the proof are in \cite[Appendix~B]{FMMS21}.

Existence and uniqueness of solutions of \eqref{eq:minimization} follows from 
the Lax-Milgram Lemma since, for $F \in L^2_{-\alpha}(\R^d \times (0,\boundY))$ and $f \in H^{-s}(\Omega)$, 
the map $U \mapsto \int_{\R^d \times (0,\boundY)} F U + \int_\Omega f \operatorname{tr} U$ 
in \eqref{eq:calFdef} extends to a bounded linear functional on ${\BL}^1_{\alpha,0,\Omega}$.
In view of \eqref{eq:lemma:properties-of-H1alpha} and the trace estimate 
\eqref{eq:L3.8-KarMel19},
the minimization problem \eqref{eq:minimization} 
admits by Lax-Milgram a unique solution $U \in {\BL}^1_{\alpha,0,\Omega}$
with the {\sl a priori} estimate 
\begin{align}
\label{eq:energy-estimate} 
\|\nabla U\|_{L^2_{\alpha}(\R^d \times \Rpos)} \leq C \left[ \|F\|_{L^2_{-\alpha}(\R^d \times (0,\boundY))} + \|f\|_{H^{-s}(\Omega)}\right] 
\end{align}
with constant $C$ dependent on $s\in (0,1)$, $\boundY > 0$, and $\Omega$. 

The Euler-Lagrange equations formally satisfied by the solution $U$ of \eqref{eq:minimization} are: 
\begin{subequations}
\label{eq:extension3D}
\begin{align}
\label{eq:extension3D-a}
 -\div (y^\alpha \nabla U) &= F  &&\text{in} \; \R^d \times (0,\infty), \\ 
\label{eq:extension3D-b}
 \partial_{n_\alpha} U(\cdot,0) & = f  &&\text{in} \; \Omega, \\
\label{eq:extension3D-c}
\operatorname{tr} U & = 0 &&\text{on $\Omega^c$},
\end{align}
\end{subequations}
where $\partial_{n_\alpha} U(x,0) = - d_s\lim_{y \rightarrow 0}  y^\alpha \partial_y U(x,y)$ 
and we implicitly extended $F$ to $\R^d \times \Rpos$ by zero.
In view of \eqref{eq:DtNoperator} together with the fractional PDE $(-\Delta)^s u = f$, 
this is a Neumann-type Caffarelli-Silvestre extension problem with an additional source $F$.

\begin{remark}\label{remk:local-extension3D}
The system \eqref{eq:extension3D} is understood in a weak sense, i.e., 
to find $U \in {\BL}^1_{\alpha,0,\Omega}$ such that 
\begin{equation}
\label{eq:weak-form}
\forall V \in {\BL}^1_{\alpha,0,\Omega} \colon \quad 
b(U,V) = \int_{\R^d \times \Rpos}  F V \,dx\,dy + \int_\Omega f \operatorname{tr} V\,dx. 
\end{equation}
Due to \eqref{eq:lemma:properties-of-H1alpha}, the integral $\int_{\R^d \times \Rpos}  F V \,dx\,dy $ is well-defined.

%
\eremk
\end{remark}
\section{Solution regularity for the CS extension}
\label{sec:RegExt}
As in \cite{FMMS21}, 
we prove analytic regularity of solutions of \eqref{eq:intro-eq} in polyhedral $\Omega\subset \R^3$
via local (higher order) regularity results for solutions to the Caffarelli-Silvestre extension problem 
in Section~\ref{sec:CSExt-sub}.
These were obtained in \cite[Sec.3]{FMMS21} for general space dimension $d \geq 2$.
We re-state these for further reference for $d=3$.

\subsection{Global regularity: a shift theorem}
\label{sec:GlRegShThm}
The following lemma provides additional regularity of the extension problem in the $x$--direction. 
Its proof is based on the difference quotient technique developed in \cite{Savare}, and was already
    used in our analysis in two spatial variables \cite{FMMS21} and in \cite{BN21} 
to establish a regularity shift in Besov scales for the Dirichlet fractional Laplacian.

For functions $U$, $F$, $f$, it is convenient to introduce the abbreviation
\begin{equation}
\label{eq:CUFf} 
N^2(U,F,f)\coloneqq  \|\nabla U\|_{L^2_\alpha(\R^d \times \Rpos)} 
\left(
\|\nabla U\|_{L^2_{\alpha}(\R^d \times \Rpos)} 
+ 
\|F\|_{L^2_{-\alpha}(\R^d \times (0,\boundY))}
+ 
\|f\|_{H^{1-s}(\Omega)}\right). 
\end{equation}
In view of the {\sl a priori} estimate \eqref{eq:energy-estimate}, 
we have the simplified bound (with updated constant $C$)
\begin{equation}
\label{eq:CUFf-simplified}
N^2(U,F,f) 
\leq 
C \left( \|f\|^2_{H^{1-s}(\Omega)} + \|F\|^2_{L^2_{-\alpha}(\R^d\times(0,\boundY))} \right).
\end{equation}

\begin{lemma}
\label{lem:regularity3D}
Let $\Omega \subset \R^3$ be a bounded Lipschitz domain, and
let $B_{\widetilde R} \subset \R^3$ be a ball with 
$\Omega \subset B_{\widetilde R}$. 
For $t \in [0,1/2)$, 
there is $C_t > 0$ (depending only on $s$, $t$, $\Omega$, $\widetilde R$, and $\boundY$) 
such that for $f\in C^\infty(\overline{\Omega})$,
$F \in L^2_{-\alpha}(\R^{3} \times (0,\boundY))$ 
the solution $U$ of \eqref{eq:minimization} satisfies 
\begin{align*}
\int_{\Rpos} y^\alpha \norm{\nabla U(\cdot, y)}_{H^t(B_{\widetilde R})}^2 dy 
\leq C_t N^2(U,F,f)
\end{align*}
with $N^2(U,F,f)$ given by \eqref{eq:CUFf}. 
\end{lemma}
This is \cite[Lemma 3.3]{FMMS21} with $d=3$.
\subsection{Caccioppoli inequalities for the CS extension} 
\label{sec:IntReg}
Our regularity will be based on Caccioppoli inequalities 
for solutions to the extension problem \eqref{eq:extension3D}.
These inequalities were derived in \cite{FMMS21}, 
but we also require them for some more general cases of tangential derivatives. 
Roughly speaking, they imply quantitative control of 
second order derivatives of $U$ on some local set (balls or sets introduced below)
in terms of first order derivatives on a (slightly) enlarged set.

\begin{definition}[Half ball, wedge]
  We call the intersection between a ball and a half
  space whose boundary passes through the center of the ball a \emph{half ball}.
  
  We call the intersection between a ball and two non-identical half
  spaces with boundaries passing through the center of the ball a \emph{wedge}.
\end{definition}

\begin{lemma}[Caccioppoli inequalities]
\label{lem:CaccType3D} 
Let $B_R(x_0)$ be an open ball with radius $R>0$ 
centered at $x_0 \in \overline{\Omega}\setminus \cV$.
Let $R>0$ be so small that
  \begin{enumerate}[label=(\roman*)]
  \item\label{item:ball} $B_R(x_0)\subset \Omega$, if $x_0 \in \Omega$;
  \item\label{item:halfball} $B_R(x_0) \cap \Omega$ is a half ball, if $x_0 \in \bff$;
  \item\label{item:wedge} $B_R(x_0) \cap \Omega$ is a wedge, if $x_0 \in \bfe$.
  \end{enumerate}
%
For $\theta \in (0,\infty]$ and $c \in (0,1]$ denote by 
$B_{cR}^\theta := (B_{cR}(x_0) \cap \Omega)\times (0,\theta) \subset \R^3 \times \R^+$ 
the corresponding concentrically scaled and extended ball/half-ball/wedge, respectively.

Let $U$ satisfy \eqref{eq:extension3D}
with given data $f$ and $F$ with 
$\supp(F)\subset \R^{3}\times [0, \boundY]$ and let $\theta'>\theta$.

Then, for $\bullet \in \{x_i: i=1,2,3\}$ in case \ref{item:ball},
$\bullet\in\{\fipar: i=1,2\}$ in case \ref{item:halfball}, and $\bullet=\epar$ in
case \ref{item:wedge}, 
there is $C_{\rm int} > 0$ independent of $R$ and $c,\theta,\theta'$ such that
\begin{align}
\label{eq:CaccType3D} 
  \norm{D_{\bullet}(\nabla U)}^{2}_{L^2_\alpha(B_{cR}^\theta)} 
&\leq \CacInt^{2} \Big( (((1-c)R)^{-2}+(\theta'-\theta)^{-2}) \norm{\nabla U}_{L^2_\alpha(B_R^{\theta'})}^{2}
\nonumber \\ 
&\qquad\quad + \norm{D_{\bullet} f}^{2}_{L^2(B_R)} + \norm{F}^{2}_{L^2_{-\alpha}(B_R^+)}\Big). 
\end{align}
\end{lemma}
\begin{proof}
  We use a cut-off function $\zeta = \zeta(x,y)$ with $0\leq \zeta \leq 1$ and product structure 
\begin{align*}
\zeta(x,y) = \zeta_x(x) \zeta_y(y), \qquad \zeta_x \in C_0^\infty(B_R), \quad \zeta_y \in C_0^\infty(-\theta^\prime,\theta^\prime).
\end{align*}
Here, $\zeta_x$ is such that 
$\zeta_x \equiv 1$ on $B_{cR}$ as well as 
  $\|\nabla \zeta_x\|_{L^\infty(B_R)} \leq C_\zeta ((1-c)R)^{-1}$ for some $C_\zeta > 0$  independent of $c$, $R$.  
Similarly, $\zeta_y$ satisfies $\zeta_y \equiv 1$ on $(-\theta ,\theta)$ as well as 
$\|\partial_y^j \zeta_y\|_{L^\infty(-\theta^\prime,\theta^\prime)} \leq C_\zeta {(\theta^\prime-\theta )}^{-j}$ for $j\in\{0,1\}$ 
with a constant $C_{\zeta}$ independent of $R$, $\theta$, $\theta'$. 
Hence $\|\nabla \zeta\|_{L^\infty(\R^3 \times \R_+)} \lesssim ((1-c) R)^{-1} + (\theta'-\theta)^{-1}$. 

Let $e_{\bullet}$ be the already defined unit vectors for $\bullet \in \{\fipar,\epar\}$ and $e_{x_i}$ be the unit vector in the $x_i$-coordinate. 
Let $\tau\in \R \backslash\{0\}$.
We define 
the difference quotient $D_{\bullet}^\tau$ as the operator such that,
for all 
$w:\R^3\times\R\to \R$,
\begin{align*}
(D_{\bullet}^\tau w)(x, y) \coloneqq  \frac{w(x+\tau e_{\bullet}, y)-w(x, y)}{\tau}, \qquad \forall x\in \R^3, y\in \R^+.
\end{align*}
We recall that by, e.g., \cite[Sec.~{6.3}]{evans98}, we have uniformly in $\tau$
\begin{equation}
  \label{eq:Dtau-partial}
  \|D^\tau _{\bullet} v\|_{L^2_{\alpha} (\R^3 \times \Rpos)} \lesssim \|\nabla v\|_{L^2_{\alpha}(\R^3\times\Rpos)}.
\end{equation}
For $|\tau|$ sufficiently small, consider the function $V = D_{\bullet}^{-\tau}(\zeta^2 D_{\bullet}^\tau U)$.
We claim $V\in\BLzero$, i.e.,
\begin{equation*}
  \tr V = 0\text{ on }\Omega^c, \qquad
V\in L^2_{loc}(\R^3\times\Rpos), \qquad
\nabla V \in L^2_\alpha(\R^3\times \Rpos).
\end{equation*}
The first property is true as long as $\tau$ is small
enough, due to the compact support of $\zeta_x$ in $B_R\subset \Omega$.
The second property follows from $\zeta\in L^\infty(\R^3\times\Rpos)$ and $V\in L^2_{loc}(\R^3\times\Rpos)$.
To show the third one, note that derivatives commute with the difference
quotient operator. It follows that
\begin{equation*}
  \partial_y V =  D_{\bullet}^{-\tau}(\zeta^2 D_{\bullet}^\tau\partial_y U)). 
\end{equation*}
Hence, $\partial_y V\in L^2_\alpha(\R^3\times\Rpos)$ since $\partial_yU \in
L^2_\alpha(\R^3\times \Rpos)$ and $\zeta$ is bounded.

Similarly, for any $j\in \{1,2,3\}$,
\begin{equation*}
  \partial_{x_j} V =  2D_{\bullet}^{-\tau}\big(\zeta(\partial_{x_j}\zeta)D_{\bullet}^{\tau} U\big) +  D_{\bullet}^{-\tau}(\zeta^2 D_{\bullet}^{\tau}\partial_{x_j} U) \eqqcolon (I) + (II).
\end{equation*}
We have
\begin{equation*}
  (I) = \frac{2}{\tau}\bigg[\big(\zeta\partial_{x_j}\zeta\big)(x-\tau e_{\bullet}, y) D_{\bullet}^{-\tau} U +  \big(\zeta\partial_{x_j}\zeta\big)(x,y) D_{\bullet}^{\tau} U\bigg].
\end{equation*}
Using the boundedness of $\zeta\partial_{x_j}\zeta$ and
since $D_{\bullet}^{-\tau} U\in L^2_\alpha(\R^3\times\Rpos)$ and 
$D_{\bullet}^{\tau} U\in L^2_\alpha(\R^3\times\Rpos)$ by \eqref{eq:Dtau-partial}, we obtain
that $(I)\in L^2_\alpha(\R^3\times\Rpos)$. In addition, by the boundedness of
$\zeta$ and since $U\in \BLzero$ implies $\partial_{x_j}U \in
L^2_\alpha(\R^3\times\Rpos)$,
we also obtain $(II)\in L^2_\alpha(\R^3\times\Rpos)$. We conclude that $\nabla
V\in L^2_{\alpha}(\R^3\times\Rpos)$. This implies $V\in\BLzero$.

We can therefore choose $V$ as a test function 
in the weak formulation of \eqref{eq:extension3D} 
and calculate
\begin{align*}
\operatorname{tr} V = -\frac{1}{\tau^2} \Big(\zeta_x^2(x-\tau e_{\bullet})(u(x)-u(x-\tau e_{\bullet}))+\zeta_x^2(x)(u(x)-u(x+\tau e_{\bullet}))\Big) = D_{\bullet}^{-\tau}(\zeta_x^2 D_{\bullet}^\tau u).
\end{align*}

Integration by parts in \eqref{eq:extension3D} tested with $V$
over $\R^3 \times \Rpos$ 
and 
using that the Neumann trace (up to the constant $d_s$ from \eqref{eq:DtNoperator})
realizes the fractional Laplacian gives
\begin{align*}
&\int_{\R^3 \times \Rpos} F V \, dx \, dy - \frac{1}{d_s} \int_{\R^3} (-\Delta)^s u  \operatorname{tr}V \, dx = 
\int_{\R^3 \times \Rpos} y^{\alpha} \nabla U \cdot\nabla V  dx\, dy \\ 
&\qquad\qquad\qquad= \int_{\R^3 \times \Rpos} D_{\bullet}^\tau (y^\alpha \nabla U) \cdot \nabla (\zeta^2 D_{\bullet}^\tau U) \,dx\, dy \\ 
&\qquad\qquad\qquad=  \int_{B_R^+} y^\alpha D_{\bullet}^\tau (\nabla U) \cdot \left(\zeta^2 \nabla D_{\bullet}^\tau U + 2\zeta \nabla \zeta D_{\bullet}^\tau U\right) dx \, dy\\
&\qquad\qquad\qquad =
  \int_{B_R^+} y^\alpha \zeta^2  D_{\bullet}^\tau (\nabla U) \cdot D_{\bullet}^\tau (\nabla U)\, dx\,dy + 
  \int_{B_R^+} 2 y^\alpha\zeta \nabla\zeta \cdot D_{\bullet}^\tau (\nabla U)  D_{\bullet}^\tau U\, dx\, dy. 
\end{align*}

Using the equation $(-\Delta)^s u = f$ on $\Omega$, Young's inequality, 
and the Poincar\'e inequality together with the trace estimate \eqref{eq:L3.8-KarMel19}, 
we get the existence of constants $C_j>0$, $j \in \{1,\dots,5\}$, such that
\begin{align*}
  \norm{\zeta D_{\bullet}^\tau(\nabla U)}_{L^2_\alpha(B_R^+)}^2
& \leq 
    C_1 \bigg(
\abs{  \int_{B_R^+} y^\alpha \zeta \nabla\zeta \cdot  D_{\bullet}^\tau (\nabla U) D_{\bullet}^\tau U \, dx \, dy} + 
  \abs{\int_{\R^3 \times \Rpos} F \; D_{\bullet}^{-\tau} \zeta^2 D_{\bullet}^\tau U \, dx \, dy}  \\
  & \qquad+ \abs{\int_{\R^3} D_{\bullet}^{\tau} f  \zeta_x^2 D_{\bullet}^\tau u \, dx}\bigg)\\
  &\leq
  \frac{1}{4}  \norm{\zeta D_{\bullet}^\tau(\nabla U)}_{L^2_\alpha(B_R^+)}^2 +  
  C_2 \bigg(\norm{\nabla \zeta}_{L^{\infty}(B_R^+)}^2\norm{D_{\bullet}^\tau U}_{L^2_\alpha(B_R^{\theta'})}^2 \\
 & \qquad +  \norm{F}_{L^2_{-\alpha}(B_R^+)} \|\nabla (\zeta^2 D^\tau_{\bullet} U)\|_{L^2_{\alpha}(B^+_R)} 
 + \norm{\zeta_x D_{\bullet}^\tau f}_{H^{-s}(\Omega)} \norm{\zeta_x D^\tau_{\bullet} u}_{H^s(\R^3)}  \bigg)\\
  & \leq
\frac{1}{2}  \norm{\zeta D_{\bullet}^\tau(\nabla U)}_{L^2_\alpha(B_R^+)}^2 
+  
C_3 \bigg(\|\nabla \zeta\|_{L^\infty(B_R^+)}^2 \|\nabla U\|_{L^2_{\alpha}(B^{\theta'}_R)}^2 + \|F\|^2_{L^2_{-\alpha}(B^+_R)} 
 \\
&\qquad\qquad + \norm{\zeta_x D_{\bullet}^\tau f}_{H^{-s}(\Omega)} \abs{\zeta_x D^\tau_{\bullet} u}_{H^s(\R^3)}\bigg) 
\\
  \overset{\eqref{eq:L3.8-KarMel19}}
  & \leq
\frac{1}{2}  \norm{\zeta D_{\bullet}^\tau(\nabla U)}_{L^2_\alpha(B_R^+)}^2 
  +  
C_4 \bigg(\|\nabla \zeta\|_{L^\infty(B_R^+)}^2 \|\nabla U\|_{L^2_{\alpha}(B^{\theta'}_R)}^2 + \|F\|^2_{L^2_{-\alpha}(B^+_R)}  
\\
&\qquad\qquad + \norm{\zeta_x D_{\bullet}^\tau f}_{H^{-s}(\Omega)} 
   \norm{\nabla (\zeta D^\tau_{\bullet} U)}_{L^2_{\alpha}( \R^3 \times \Rpos)} \bigg)
\\
 &\leq 
   \frac{3}{4}  \norm{\zeta D_{\bullet}^\tau(\nabla U)}_{L^2_\alpha(B_R^+)}^2
   \\ & \qquad \qquad
+  C_5 \bigg(\|\nabla \zeta\|_{L^\infty(B_R^+)}^2 \|\nabla U\|_{L^2_{\alpha}(B^{\theta'}_R)}^2 + \|F\|^2_{L^2_{-\alpha}(B^+_R)} 
 + \norm{\zeta_x D_{\bullet}^\tau f}^2_{H^{-s}(\Omega)}  \bigg).  
\end{align*}
Absorbing the first term of the right-hand side in the left-hand side and taking the limit $\tau \rightarrow 0$, 
we obtain the sought inequality for the second derivatives 
since $\norm{\nabla \zeta}_{L^{\infty}(B_R^+)} \lesssim ((1-c)R)^{-1}+ (\theta'-\theta)^{-1}$.
We conclude using
$\|\zeta_x D_{\bullet} f\|_{H^{-s}(\Omega)} \leq C_{\rm loc} \|D_{\bullet} f\|_{L^2(B_R)}$ 
for some $C_{\rm loc} > 0$ independent of $R$, $c$, and $f$.
\end{proof}

The Caccioppoli inequality in Lemma~\ref{lem:CaccType3D} 
can be iterated on concentric balls to provide control of higher order derivatives 
by lower order derivatives locally. 

\begin{corollary}[High order interior Caccioppoli inequality]
\label{cor:CaccHighInt} 
Let $B_R(x_0)\subset \Omega$ be an open ball with radius $R>0$ centered at $x_0 \in \Omega$. 
For $\theta \in (0,\infty]$ and $c \in (0,1]$ denote by $B_{cR}^\theta := B_{cR}(x_0)\times (0,\theta)$ 
the corresponding concentrically scaled and extended ball. 
Let $U$ satisfy \eqref{eq:extension3D}
with given data $f$ and $F$ with $\supp(F)\subset \R^3\times [0, \boundY]$ and let $\theta'>\theta$.

Then, 
there is $\gamma > 0$ 
such that 
for all $\beta \in \mathbb{N}_0^3$ 
we have with $p = \betam  $ 
\begin{multline}
  \label{eq:CaccHighBound-int}
\norm{\dbeta\nabla U}^2_{L^2_\alpha(B_{cR}^\theta)} \leq
 (\gamma p)^{2p} R^{-2p}  \norm{\nabla U}^2_{L^2_\alpha(B_R^{\theta'})} \\
  + \sum_{j=1}^p (\gamma p)^{2(p-j)} R^{2(j-p)}\left(\max_{\substack{\etam=j\\ \eta\leq \beta}}\norm{\deta f}^2_{L^2(B_R)} 
  +\max_{\substack{\etam=j-1\\ \eta\leq \beta}} \norm{\deta F}^2_{L^2_{-\alpha}(B_{R}^+)}\right).
\end{multline}
\end{corollary}
\begin{proof}
We start by noting that the case $p = 0$ is trivially true since 
empty sums are zero and $0^0 = 1$. 
For $p \ge 1$, we fix a multi index $\beta$ such that $\betam = p$. 
  As the $x$-derivatives commute with the differential operator in
  \eqref{eq:extension3D}, we have that $\dbeta U$ solves 
equation \eqref{eq:extension3D} with data $\dbeta F$ and $\dbeta f$. 
For given $c>0$ and $0<\theta<\theta'$, let
\begin{align*}
  c_i &= c + (i-1)\frac{1-c}{p}, \quad 
  \theta_i = \theta + (i-1)\frac{\theta^\prime-\theta}{p}, \qquad i=1, \dots, p+1. 
\end{align*}
Then, we have $c_{i+1}R-c_iR = \frac{(1-c)R}{p}$, $c_1 R = cR$, and 
$c_{p+1} R = R$ as well as $\theta_{i+1}-\theta_i = \frac{\theta'-\theta}{p}$, $\theta_1=\theta$, 
and $\theta_{p+1} = \theta'$.  As $R \leq \operatorname{diam} \Omega$, 
we obtain 
\begin{align*}
(\theta_{i+1}-\theta_i)^{-2} + (c_{i+1} R - c_{i} R)^{-2} \leq C p^2 R^{-2}/(1-c)^2
\end{align*} 
with a constant $C > 0$ depending only on $\Omega$, $\theta$, $\theta'$. 
For ease of notation and without loss of generality, we assume that $\beta_1>0$. 
Applying Lemma~\ref{lem:CaccType3D} iteratively on the sets $B_{c_iR}^{\theta_i}$ for $i>1$ provides 
\begin{align*}
 & \norm{\dbeta \nabla U}^2_{L^2_\alpha(B_{c R}^{\theta})} \\
& \quad \leq \CacInt^{2} 
           \left(\frac{p^2}{(1-c)^2}R^{-2} \norm{\partial_x^{(\beta_1-1, \beta_2)}\nabla U}^2_{L^2_\alpha(B_{c_2 R}^{\theta_2})}
 + C_{\rm loc}^2 \norm{\dbeta  f}^2_{L^2(B_{c_2 R})} + \norm{\partial_x^{(\beta_1-1, \beta_2)}F}^2_{L^2_{-\alpha}(B_{c_2 R}^+)}  \right) \\
  & \quad \leq \left(\frac{\CacInt p }{(1-c)}\right)^{2p} R^{-2p} \norm{\nabla U}^2_{L^2_\alpha(B_{R}^{\theta'})}
    + C_{\rm loc}^2 \sum_{j=1}^p\left(\frac{\CacInt p }{(1-c)}\right)^{2p-2j}R^{-2p+2j}\max_{\etam=j}\norm{\deta  f}^2_{L^2(B_{c_{p-j+2} R})} \\
  &\qquad + \sum_{j=0}^{p-1}\left(\frac{\CacInt p }{(1-c)}\right)^{2p-2j-2}R^{-2p+2j+2}\max_{\etam=j}\norm{\deta F}^2_{L^2_{-\alpha}(B_{c_{p-j+1} R}^+)}. 
\end{align*}
Choosing $\gamma = \max(C_{\rm loc}^2, 1)\CacInt/(1-c)$ concludes the proof.
\end{proof}

The same arguments also apply to the other cases 
in the statement of Lemma~\ref{lem:CaccType3D} for sets near faces and edges.

\begin{corollary}[High order boundary Caccioppoli inequality on $\bff$] 
  \label{cor:CaccHighBound-f}
\mbox{ } \newline
Let $\bff\in \cF$ be an open face of $\partial\Omega$ and $x_0 \in \bff$. 
For $R>0$, let $B_R(x_0)\cap \Omega$ be an open half-ball. 
For $\theta \in (0,\infty]$ and $c \in (0,1]$ denote by $B_{cR}^\theta := (B_{cR}(x_0) \cap \Omega)\times (0,\theta)$ 
the corresponding concentrically scaled and extended half-ball. 
Let $U$ satisfy \eqref{eq:extension3D}
with given data $f$ and $F$ with $\supp(F)\subset \R^3\times [0, \boundY]$ and let $\theta'>\theta$.

Then, 
there is $\gamma>0$ such that for every for all $\betapar \in \N^2_0$ with $p = |\betapar|$,
\begin{multline}
  \label{eq:CaccHighBound-f}
\|\Dparf^{\betapar}\nabla U\|^2_{L^2_\alpha(B_{cR}^\theta)}  \leq 
  (\gamma p)^{2p} R^{-2p} \|\nabla U\|^2_{L^2_\alpha(B_R^{\theta'})}  \\
  + \sum_{j=1}^p (\gamma p)^{2(p-j)} R^{2(j-p)}\left(\max_{\substack{\etam=j\\\eta\leq \betapar}}\|\Dparf^{\eta} f\|^2_{L^2(B_R)} 
  +\max_{\substack{\etam=j-1\\\eta\leq \betapar}} \|\Dparf^{\eta} F\|^2_{L^2_{-\alpha}(B_{R}^+)}\right).
\end{multline}
\end{corollary}

\begin{corollary}[High order boundary Caccioppoli inequality on $\bfe$]
  \label{cor:CaccHighBound-e}
\mbox{ } \newline
Let $\bfe \in \cE$ be an open edge of $\partial\Omega$ and $x_0 \in \bfe$. 
For $R>0$, let $B_R(x_0)\cap \Omega$ be an open wedge. 
For $\theta \in (0,\infty]$ and $c \in (0,1]$ denote by $B_{cR}^\theta := (B_{cR}(x_0) \cap \Omega)\times (0,\theta)$ 
the corresponding concentrically scaled and extended wedge. 
Let $U$ satisfy \eqref{eq:extension3D}
with given data $f$ and $F$ with $\supp(F)\subset \R^3\times [0, \boundY]$ and let $\theta'>\theta$.

Then, there is $\gamma>0$ such that for every $p \in \N_{0}$ 
  \begin{align}
    \label{eq:CaccHighBound-e}
    \|\Dpare^p\nabla U\|^2_{L^2_\alpha(B_{cR}^\theta)}
    & \leq (\gamma p)^{2p} R^{-2p} \|\nabla U\|^2_{L^2_\alpha(B_R^{\theta'})}  \\
    \nonumber
    & \qquad + \sum_{j=1}^p (\gamma p)^{2(p-j)} R^{2(j-p)}\left(\|\Dpare^j f\|^2_{L^2(B_R)}
    + \|\Dpare^{j-1} F\|^2_{L^2_{-\alpha}(B_{R}^+)}\right).
  \end{align}
\end{corollary}

\section{Local tangential regularity for the CS extension} 
\label{sec:LocTgReg}

Employing additional regularity of $U$, which was shown in 
Lemma~\ref{lem:regularity3D}, the term $\norm{\nabla U}_{L^2_{\alpha}(B_R^+)}$
in \eqref{eq:CaccHighBound-int} -- \eqref{eq:CaccHighBound-e} is small for $R\rightarrow 0$. 
This is the made precise in the following lemma, 
which is the exact analog of the corresponding statement in dimension $d=2$ 
near edges \cite[Lem. 4.3]{FMMS21}.
\begin{lemma}\label{lem:estH13D}
For $t \in [0,1/2)$, there exists $\Creg > 0$ 
(depending only on $t$ and $\Omega$) 
such that the solution $U$ of \eqref{eq:minimization} satisfies 
\begin{align}
\label{eq:lem:estH13D-5}
\|r^{-t}_{\partial\Omega} \nabla U \|^2_{L^2_{\alpha}(\Omega^+)} 
\leq 
C_{\rm reg} C_t N^2(U,F,f)
\end{align}
with the constant $C_t>0$ from Lemma~\ref{lem:regularity3D} and $N^2(U,F,f)$ given by \eqref{eq:CUFf}. 
\end{lemma}
Lemma~\ref{lem:regularity3D} provides global regularity
for the solution $U$ of \eqref{eq:extension3D}. 
For all $R, \boundY>0$ and $x_0\in \R^3$, 
let $B_R^\boundY(x_0) \coloneqq B_R(x_0)\times (0, \boundY)$.
We introduce, 
for any set $B_R^\boundY\subset \R^3\times \Rpos$ and any $p\in \N_0$,
\begin{equation}
  \label{eq:Ntilde}
  \tNp_{B_R^\boundY}(F,f) \coloneqq 
  \sum_{j=1}^{p+1} (\gamma p)^{-2j}\bigg(3^j\max_{\betam =j}\|\dbeta f\|^2_{L^2(B_R)}
  + 3^{j-1}\max_{\betam=j-1}\|\dbeta F\|^2_{L^2_{-\alpha}(B_R^\boundY)}\bigg).
\end{equation}
We derive localized versions of Lemma~\ref{lem:regularity3D} 
for tangential derivatives of $U$ at the boundary. 
Their proofs are minor variations of arguments in the proof of \cite[Lemma 4.4]{FMMS21};
we present the details here for completeness. 

\begin{lemma}[High order localized shift theorem near a face or an edge]
\label{lem:localhighregularity3D-ef}
Let $U$ be the solution of \eqref{eq:minimization}.  Let $\bfs\in \cE\cup\cF$.
Let $x_0 \in \bfs$. 
Let $R\in (0,1/2]$, $c\in(0,1)$, and assume that 
$B_R(x_0) \cap \Omega$ is a half ball (if $\bfs\in \cF$) or a wedge (if $\bfs\in \cE$). 

Then,
for $t \in [0,1/2)$, there is $C > 0$ independent of $R$ and $x_0$
 such that, for all $\beta\in \N$ (if $\bfs\in \cE$) or $\beta \in \N_0^2$ (if
 $\bfs \in \cF$), with $\betam \eqqcolon p\in \N_0$, 
\begin{equation}
\label{eq:lem:localhighregularity3D-5}
  \| r_{\partial \Omega}^{-t}\Dpars^\beta \nabla U \|^2_{L^2_\alpha(B^{\boundY/2}_{cR})}
   \leq C  R^{-2p-1}(\gamma p)^{2p}(1+\gamma p) \left(\| \nabla U\|_{L^2_\alpha(B_R^\boundY)}^{2} 
        +  R^{s+1}\tNp_{B_R^\boundY}(F, f)  \right),
\end{equation}
where $\gamma$ is the constant in Corollary~\ref{cor:CaccHighBound-e} or \ref{cor:CaccHighBound-f}. 
\end{lemma}
\begin{proof}
  Let $\tc = (c+1)/2 \in (c, 1)$.
  Let $\eta_x \in C^\infty_0(B_{\tc R}(x_0))$ with $\eta_x\equiv 1 $ on
  $B_{cR}(x_0)$, $\eta_y \in C^\infty_0(-\boundY,\boundY)$ with $\eta_y \equiv 1$ on $(-\boundY/2,\boundY/2)$ and 
  $\|\nabla^j \eta_x\|_{L^\infty(B_R(x_0))} \leq C_\eta R^{-j}$, $j \in \{0,1,2\}$ as well as 
  $\|\partial_y^j \eta_y\|_{L^\infty(-\boundY,\boundY)} \leq C_\eta \boundY^{-j}$, $j \in \{0,1,2\}$, 
  with a constant $C_\eta > 0$ independent of $R$ and $\boundY$.
  Let $\eta(x,y):= \eta_x(x) \eta_y(y)$. 

We denote $\kappa=1$ if $\bfs\in\cE$  and $\kappa=2$ if $\bfs\in \cF$ (so that
$\beta\in \N^\kappa_0$). 
We abbreviate $\Uparbeta\coloneqq  \Dpars^\beta U$, $\tUbeta\coloneqq  \eta \Dpars^\beta U$, 
$\Fparbeta = \Dpars^\beta F$, and $\fparbeta  = \Dpars^\beta f$. 
Throughout the proof we will use the fact that, 
for all $j\in \N$ and all sufficiently smooth functions $v$, 
we have 
\begin{equation*}
  \max_{\etam=j}|\Dpars^\eta v | \leq 3^{j/2}\max_{\betam=j} |\dbeta v|.
\end{equation*}
We also note that the assumptions on $\eta(x,y) = \eta_x(x) \eta_y(y)$ imply the
existence of $\tilde C_\eta > 0$ 
(which absorbs the dependence on $\boundY$ and $c$ that we do not further track) 
such that 
\begin{equation}
\|\nabla^j_{x} \partial_y^{j'} \eta\|_{L^\infty(\R^3 \times \R)} \leq \tilde C_\eta R^{-j}, \qquad j \in \{0,1,2\}, j' \in \{0,1,2\}. 
\end{equation}
\paragraph{\bf Step 1} (Localization of the equation). 
Using that $U$ solves the extension problem \eqref{eq:extension3D}, we obtain that the function 
$\tUbeta = \eta \Uparbeta$ satisfies in $\Omega\times (0,\infty)$ 
the equation
\begin{multline*}
   \tFbeta : = \div (y^\alpha \nabla \tUbeta)
  \\
  \begin{aligned}[t]
  &=   y^\alpha \operatorname{div}_x (\nabla_x \tUbeta) + \partial_y(y^\alpha \partial_y\tUbeta) 
  \\ &
    \begin{multlined}[][\arraycolsep]
      =  y^\alpha\left((\Delta_x \eta) \Uparbeta + 2 \nabla_x \eta \cdot \nabla_x
        \Uparbeta +\eta \Delta_x \Uparbeta\right) + \eta\partial_y(y^\alpha
      \partial_y\Uparbeta)
      \\
      + \partial_y (y^\alpha \Uparbeta \partial_y \eta ) + y^\alpha \partial_y \Uparbeta \partial_y \eta 
    \end{multlined}
  \\
  &
    = y^\alpha\left((\Delta_x \eta) \Uparbeta + 2 \nabla_x \eta \cdot \nabla_x \Uparbeta\right) 
+ \partial_y (y^\alpha \Uparbeta \partial_y \eta ) + y^\alpha \partial_y \Uparbeta \partial_y \eta 
+ \eta   \operatorname*{div} (y^\alpha \nabla \Uparbeta) 
    \\
  & = y^\alpha\left((\Delta_x \eta) \Uparbeta + 2 \nabla_x \eta \cdot \nabla_x \Uparbeta\right)
+ \partial_y (y^\alpha \Uparbeta \partial_y \eta ) + y^\alpha \partial_y \Uparbeta \partial_y \eta
+ \eta \Fparbeta 
  \end{aligned}
\end{multline*}
as well as the boundary conditions 
\begin{align*}
\partial_{n_\alpha} \tUbeta(\cdot,0) & = \eta(\cdot,0) \Dpars^{\beta} f  \eqqcolon  \tfbeta   &&\text{ on $\Omega$}, \\
\operatorname{tr}\tUbeta & = 0  &&\text{ on $\Omega^c$.}
\end{align*}
By the support properties of the cut-off function $\eta$, we have $\supp \tFbeta
\subset \overline{B_{\tc R}}(x_0) \times [0,\boundY] $.
Using Lemma~\ref{lem:regularity3D}, 
for all $t \in [0,1/2)$, there is a $C_t>0$ such that
\begin{equation}
  \label{eq:reg2D_in_highord}
  \int_{\Rpos}y^\alpha \|\nabla\tUbeta(\cdot, y)\|^2_{H^t(B_{\widetilde{R}})}dy \leq C_t N^2(\tUbeta, \tFbeta, \tfbeta),
\end{equation}
where $B_{\widetilde{R}}$ is a ball containing $\overline{\Omega}$.
By \eqref{eq:CUFf}, we must bound $N^2(\tUbeta, \tFbeta, \tfbeta)$, i.e., 
the quantities
$\|\nabla \tUbeta\|_{L^2_{\alpha}(\R^3\times \Rpos)}$, $\|\tFbeta\|_{L^2_{-\alpha}(\R^3 \times (0,\boundY))}$, 
and 
$\|\tfbeta\|_{H^{1-s}(\Omega)}$. 
In the following, $\gamma$ is the constant introduced in
Corollary~\ref{cor:CaccHighBound-e} or \ref{cor:CaccHighBound-f}.

\paragraph{\bf Step 2} (Estimate of $\|\nabla \tUbeta\|_{L^2_{\alpha}(\R^3\times \Rpos)}$).
Let $\tbeta\in \N^\kappa_0$ be any (multi-)index such that $\tbetam = p-1$.
We write
\begin{align}\label{eq:lem:localregularity2D-00}
\nonumber 
\|\nabla \tUbeta\|^2_{L^2_{\alpha}(\R^3\times \Rpos)} &\leq 
2\|\nabla \eta\|^2_{L^\infty(B_R^\boundY)}  \|\nabla_x \Upar^{(\tbeta)}\|^2_{L^2_{\alpha}(B_{\tc R}^\boundY)} + 2\|\eta\|^2_{L^\infty(B_{\tc R}^\boundY)} \|\nabla \Uparbeta\|^2_{L^2_{\alpha}(B^\boundY_{{\tc R}})} 
  \\ 
&\leq 2\tilde C_{\eta}^2 \left( R^{-2}
\|\nabla\Upar^{(\tbeta)}\|^2_{L^2_{\alpha}(B^\boundY_{\tc R})} + \|\nabla \Uparbeta\|^2_{L^2_{\alpha}(B^\boundY_{\tc R})}\right). 
\end{align}
We employ Corollary~\ref{cor:CaccHighBound-e} or \ref{cor:CaccHighBound-f} (with
$\tc$ instead of $c$)
to obtain for all $\beta \in \N^\kappa_0$
\begin{equation}\label{eq:lem:localregularity2D-10}
  \begin{aligned}
  \|\nabla \Uparbeta\|^2_{L^2_{\alpha}(B^\boundY_{\tc R})}
  &
    \begin{multlined}[t][\arraycolsep]
      \leq R^{-2p}(\gamma p)^{2p}\bigg(\|\nabla U\|^2_{L^2_{\alpha}(B^\boundY_{R})}
      \\ + \sum_{j=1}^p R^{2j} (\gamma p)^{-2j} \Big(\max_{\substack{\etam=j\\\eta\leq \beta}}\|\Dpars^\eta f\|^2_{L^2(B_{R})} + \max_{\substack{\etam=j-1\\\eta\leq \beta}}\|\Dpars^{\eta} F\|^2_{L^2_{-\alpha}(B_{R}^\boundY)}\Big)\bigg) 
    \end{multlined}
    \\ &
         \begin{multlined}[t][\arraycolsep]
         \leq R^{-2p}(\gamma p)^{2p} \bigg( \|\nabla U\|^2_{L^2_{\alpha}(B^\boundY_{R})} \\
   + R^{2}\sum_{j=1}^p R^{2(j-1)} (\gamma p)^{-2j}\Big(3^j\max_{\betam=j} \|\dbeta f\|^2_{L^2(B_{R})}
    + 3^{j-1}\max_{\betam=j-1} \|\dbeta F\|^2_{L^2_{-\alpha}(B_{R}^\boundY)} \Big)\bigg) 
  \end{multlined}
\\ &\leq R^{-2p} (\gamma p)^{2p}\left(\|\nabla U\|^2_{L^2_{\alpha}(B^\boundY_{R})} + R^2 \tNp_{B^\boundY_R}(F, f) \right).
  \end{aligned}
\end{equation} 
%
For $p \in \N$, we 
apply \eqref{eq:lem:localregularity2D-10} to the $\tbeta$-derivative 
and 
exploit the estimate $(\gamma(p-1))^{-2} \leq \max\{1,\gamma^{-2}\}$ for $p>1$ to  bound
$(\gamma (p-1))^{2p-2} \tNpmone_{B_R^\boundY}(F,f) \lesssim  \max\{1,\gamma^{-2}\}(\gamma p)^{2p} \tNp_{B_R^\boundY}(F,f)$.
Consequently, 
we obtain the existence of a constant $C>0$ such that for all $p\in \N$ it holds that
(recall $|\tbeta| = p-1$) 
\begin{align}\label{eq:lem:localregularity2D-20}
  \|\nabla \Upar^{(\tbeta)}\|^2_{L^2_{\alpha}(B^\boundY_{\tc R})}
  &\leq C
    \max\{1,\gamma^{-2}\} 
    R^{-2p+2} (\gamma p)^{2p}\left(\|\nabla U\|^2_{L^2_{\alpha}(B^\boundY_{R})} + R^2 \tNp_{B^\boundY_R}(F, f) \right)
    \;.
\end{align}
Inserting \eqref{eq:lem:localregularity2D-10} and \eqref{eq:lem:localregularity2D-20} 
into \eqref{eq:lem:localregularity2D-00} 
provides the estimate
\begin{align*}
  \|\nabla \tUbeta\|^2_{L^2_{\alpha}(\R^3\times \Rpos)} \leq C
  R^{-2p} (\gamma p)^{2p}\left(\|\nabla U\|^2_{L^2_{\alpha}(B^\boundY_{R})} + R^2 \tNp_{B^\boundY_R}(F, f) \right)
\end{align*}
with a constant $C>0$ depending only on the constants $\tilde C_{\eta}$, $c$, and $\gamma$.

\paragraph{\bf Step 3} (Estimate of $\|\tFbeta\|_{L^2_{-\alpha}(\R^3 \times \Rpos)}$).
We treat the five terms appearing in $\|\tFbeta\|_{L^2_{-\alpha}(\R^3 \times \Rpos)}$ separately. 
With \eqref{eq:lem:localregularity2D-10}, we obtain
\begin{align*}
\norm{ y^\alpha\nabla_x \eta \cdot \nabla_x \Uparbeta}^2_{L^2_{-\alpha}(\R^3 \times (0,\boundY))} &= 
\norm{\nabla_x \eta \cdot \nabla_x \Uparbeta}^2_{L^2_{\alpha}(\R^3 \times \Rpos)} \leq 
C_\eta^2 \frac{1}{R^2} \norm{\nabla_x \Uparbeta}^2_{L^2_{\alpha}(B_{\tc R}^\boundY)} 
\\ 
\overset{ \eqref{eq:lem:localregularity2D-10}}&{\leq}
C R^{-2p-2} (\gamma p)^{2p}\left(\|\nabla U\|^2_{L^2_{\alpha}(B^\boundY_{R})} + R^2 \tNp_{B^\boundY_R}(F, f) \right)
\;.
\end{align*}
Similarly, we get (with $\tbetam = p-1$ again)
\begin{align*}
\norm{y^\alpha(\Delta_x \eta) \Uparbeta}^2_{L^2_{-\alpha}(\R^3 \times (0,\boundY))}
&=  \norm{(\Delta_x \eta) \Uparbeta}^2_{L^2_{\alpha}(B^\boundY_{\tc R})} \leq
  C_\eta^2 \frac{1}{R^4} \norm{\nabla \Upar^{(\tbeta)}}^2_{L^2_{\alpha}(B_{\tc R}^\boundY)}
\\ 
  \overset{\text{\eqref{eq:lem:localregularity2D-20}}}&{\leq} 
  C R^{-2p-2} (\gamma p)^{2p}\left(\|\nabla U\|^2_{L^2_{\alpha}(B^\boundY_{R})} + R^2 \tNp_{B^\boundY_R}(F, f) \right)
\;.
\end{align*}
Next, we estimate 
\begin{align*}
  \|\eta \Fparbeta\|^2_{L^2_{-\alpha} (\R^3 \times (0,\boundY))}
  \leq \|\Fparbeta\|^2_{L^2_{-\alpha} (B_{\tc R}^\boundY)} \leq 3^p \max_{\betam=p} \|\dbeta F\|_{L^2_{-\alpha} (B_{\tc R}^\boundY)}^2 \leq (\gamma p)^{2p+2} \tNp_{B_R^\boundY}(F, f). 
\end{align*}
Finally, for the term 
$\partial_y (y^\alpha \Uparbeta \partial_y \eta ) + y^\alpha \partial_y \Uparbeta \partial_y \eta$, we observe that $\partial_y \eta$ vanishes near $y = 0$ so that the weight
$y^\alpha$ does not come into play as it can be bounded from above and below by positive constants depending only on $\boundY$. We arrive at 
\begin{multline*}
 \norm{\partial_y (y^\alpha \Uparbeta \partial_y \eta ) + y^\alpha \partial_y \Uparbeta \partial_y \eta}_{L^2_{-\alpha}(\R^3 \times (0,\boundY))}^{2} 
  \\ 
       \begin{aligned}[t]
       & \leq 
C \left( \boundY^{-2} \|\Uparbeta\|_{L^2_\alpha(B_{\tc R} \times (0,\boundY))}^{2} 
  + \boundY^{-1}\|\nabla \Uparbeta\|_{L^2_\alpha(B^\boundY_{\tc R})}^{2} \right) 
\\
  \overset{\eqref{eq:lem:localregularity2D-10},\eqref{eq:lem:localregularity2D-20}}&{\leq } C_\boundY
R^{-2p-2} (\gamma p)^{2p}\left(\|\nabla U\|^2_{L^2_{\alpha}(B^\boundY_{R})} + R^2 \tNp_{B^\boundY_R}(F, f) \right)
       \end{aligned}
\end{multline*}
for suitable $C_\boundY > 0$ depending on $\boundY$. 

\paragraph{\bf Step 4} (Estimate of $\|\tfbeta\|_{H^{1-s}(\Omega)}$.)
Here, we use
Lemma~\ref{lemma:localization-fractional-norms} and $R<1/2$ together with $s<1$ to obtain 
\begin{align*}
\|\tfbeta\|_{H^{1-s}(\Omega)}^2 & 
\leq 2\Cloctwo^2 C_\eta^2 \left( 9R^{2s-2} \|\Dpars^\beta f\|^2_{L^2(B_R)} + |\Dpars^\beta f|^2_{H^{1-s}(B_R)} \right) 
  \\ &
  \leq C\Cloctwo^2C_\eta^2 R^{2s-2}\left(3^p \max_{\betam=p}\| \dbeta f\|^2_{L^2(B_R)} + 3^{p+1} \max_{\betam=p+1}\|\dbeta f\|^2_{L^2(B_R)}  \right)
  \\ &
  \leq C \Cloctwo^2C_\eta^2 R^{2s-2}(\gamma p)^{2p}(1+(\gamma p)^2) \tNp_{B_R^\boundY}(F, f) 
\end{align*}
with a constant $C>0$ depending only on $\Omega$, $s$, and $c$.


\paragraph{\bf Step 5} (Putting everything together.)
Combining the above estimates, we obtain that there exists 
a constant $C>0$ depending only on
$\tilde C_\eta$,
$\Cloctwo$, $\boundY$, $\gamma$, $\Omega$, $s$, and $c$ such that
\begin{align*}
  & N^2(\tUbeta,\tFbeta,\tfbeta)
  \\ & \quad = 
\|\nabla \tUbeta\|^2_{L^2_\alpha(\R^3 \times \Rpos)} + 
\|\nabla \tUbeta\|_{L^2_\alpha(\R^3 \times \Rpos)} \|\tFbeta\|_{L^2_{-\alpha}(\R^3\times (0,\boundY))} + 
       \|\nabla \tUbeta\|_{L^2_\alpha(\R^3 \times \Rpos)} \|\tfbeta\|_{H^{1-s}(\Omega)}   
  \\ & \quad \leq
        C \left( 1 + \gamma p R^{-1} + R^{-1}(1+\gamma p) \right)R^{-2p}(\gamma p)^{2p} \left( \|\nabla U\|_{L^2_\alpha(B_R^\boundY)}^{2} 
                 + R^{s+1}\tNp_{B_R^\boundY}(F, f) \right).
\end{align*}
Inserting this estimate in \eqref{eq:reg2D_in_highord} 
we conclude that
\begin{equation*}
  \int_{\Rpos} y^\alpha \norm{\nabla \tUbeta(\cdot,y)}_{H^t(\Omega)}^2 dy \leq C
  \left( 1 + \gamma p  \right)R^{-2p-1}(\gamma p)^{2p} \left( \|\nabla U\|_{L^2_\alpha(B_R^\boundY)}^{2} 
         + R^{s+1}\tNp_{B_R^\boundY}(F, f) \right).
\end{equation*}

\paragraph{\bf Step 6} The estimate \eqref{eq:lem:localhighregularity3D-5} 
follows from \cite[Thm.~{1.4.4.3}]{Grisvard},
which gives
\begin{equation*}
  \int_{\Rpos} y^\alpha \|r^{-t}_{\partial\Omega} \nabla \widetilde U^{(p)}(\cdot,y)\|^2_{L^2(\Omega)}\,dy 
\leq C 
\int_{\Rpos} y^\alpha \|\nabla \widetilde U^{(p)}(\cdot,y)\|^2_{H^t(\Omega)}\,dy
\end{equation*}
and from $\tUbeta = \Dpars^p U$ on $B_{cR}\times (0, \boundY/2)$ by the definition of $\eta$.
\end{proof}

The following lemma is the same of the above, but in the interior of the domain.
\begin{lemma}[High order localized shift theorem in the interior]
\label{lem:localhighregularity3D-int}
Let $U$ be the solution of \eqref{eq:minimization}.  
Let $x_0 \in \Omega$. 
Let $R\in (0,1/2]$, $c\in(0,1)$, and assume that 
$B_R(x_0) \subset \Omega$. 

Then,
for $t \in [0,1/2)$, there is $C > 0$ independent of $R$ and $x_0$
 such that, for all $\beta\in \N_0^3$, with $p = \betam \in \N_0$, 
%
\begin{equation}
\label{eq:lem:localhighregularity3D-int}
  \| r_{\partial \Omega}^{-t}\dbeta \nabla U \|^2_{L^2_\alpha(B^{\boundY/2}_{cR})}
   \leq C  R^{-2p-1}(\gamma p)^{2p}(1+\gamma p) \left(\| \nabla U\|_{L^2_\alpha(B_R^\boundY)}^{2} 
            +  R^{s+1}\tNp_{B_R^\boundY}(F, f)  \right).
\end{equation}
%
\end{lemma}
\begin{proof}
  The proof is the same as that of Lemma~\ref{lem:localhighregularity3D-ef},
  with Corollary~\ref{cor:CaccHighInt} replacing
  Corollary~\ref{cor:CaccHighBound-e} or \ref{cor:CaccHighBound-f}.
\end{proof}

\newpage
\section{Weighted $H^p$-estimates in polyhedra}
\label{sec:WghHpPolygon}
In this section, 
we derive higher order weighted regularity results, 
at first for the extension problem and finally for the fractional PDE.
The strategy is as in the two-dimensional case:
we first introduce suitable countable, locally finite coverings of the
various neighborhoods in Section~\ref{sec:Cover}.
We then obtain in each of the neighborhoods local, Caccioppoli-type regularity
shifts for the solution $U$ of the CS extension defined in Section~\ref{sec:CSExt-sub},
in Section~\ref{sec:WgtHpRegExtPbm}.
Finally, in Section~\ref{sec:WgtHpRegu}, we deduce from the estimates on $U$ 
the analytic regularity results for the solution $u$ of \eqref{eq:weakform}.
%
\subsection{Coverings}
\label{sec:Cover}
As in space dimension $d=2$, \cite{FMMS21}, 
a main ingredient in the proofs of a-priori estimates are suitable 
localizations of all the geometric neighborhoods in the partition \eqref{eq:Nghbrhoods}
of the polyhedron $\Omega$.

This is achieved by covering such neighborhoods by balls, half-balls or wedges with the following two properties: 
a) their diameter is proportional to the distance to 
lower-dimensional singular supports, i.e., vertices, edges and faces,
and
b) scaled versions of the balls/cut-balls satisfy a locally finite overlap property.

The general procedure in our construction of suitable localized coverings
of all neighborhoods is hierarchic with respect to the dimension of the singular support set: 
if $\omega_{\bullet}$ is close to only one singular component, i.e., to either one
vertex, edge or face (i.e. $\bullet \in \{\bfv,\bfe,\bff\}$), 
we use balls inscribed in $\Omega$ with radii
proportional to the distance to $\partial \Omega$. 

For $\omega_{\bullet}$ close to two singular components of $\partial \Omega$, 
i.e., $\bullet \in \{\bfv\bfe,\bfv\bff,\bfe\bff\}$, 
we localize at first with half-balls (in case of neighborhoods close to faces) centered on $\bff$ 
in direction of the edge/vertex or wedges (in case of $\omega_{\bfv\bfe}$) in direction of the vertex.
Then, the half-balls/wedges are localized again using balls centered in $\Omega$ 
in direction of the face/edge (implicitly 
done in Lemma~\ref{lemma:regularity-W} and Lemma~\ref{lemma:regularity-halfball}).  

For $\omega_{\bullet}$ situated simultaneously close to three singular components of $\partial\Omega$, 
i.e. belonging to vertex-edge-face-neighborhoods, 
we first localize with wedges centered on the edge in direction of the vertex, 
then with half-balls centered on the face in direction of the edge,
and 
finally with balls centered in $\Omega$ in direction of the face.

As in the two-dimensional case \cite[Lemma 5.1]{FMMS21},  
we work with local estimates obtained from Besicovitch's Covering Theorem.
\begin{lemma}[{\cite[Lem.~{A.1}]{melenk-wohlmuth12}, \cite[Lem.~{A.1}]{horger-melenk-wohlmuth13}}]
\label{lemma:MW}
Let $\omega\subset {\mathbb R}^d$ be bounded, open and let $M \subset \partial\omega$ be closed, and nonempty.
Fix $c$, $\zeta \in (0,1)$ such that $1 -c (1+\zeta) \eqqcolon c_0 > 0$. 
For each $x \in \omega$, let $B_x\coloneqq  \overline{B}_{c \operatorname{dist}(x,M)}(x)$ be the closed ball of 
radius $c \operatorname{dist}(x,M)$ centered at $x$, 
and let
$\widehat{B}_x\coloneqq  \overline{B}_{(1+\zeta) c \operatorname{dist}(x,M)}(x)$ 
be the scaled closed ball of radius $(1+\zeta) c \operatorname{dist}(x,M)$ centered at $x$.  

Then, there is a countable set 
$(x_i)_{i \in {\mathcal I}} \subset  \omega$ 
(for some suitable index set ${\mathcal I} \subset {\mathbb N}$) 
and a number $N \in {\mathbb N}$ 
depending solely on $d$, $c$, $\zeta$ with the following properties:
\begin{enumerate}
\item 
(covering property) $\bigcup_{i} B_{x_i} \supset \omega$. 
\item
(finite overlap) $\operatorname{card}\{i\,|\, x \in \widehat{B}_{x_i} \} \leq N$
for all $x \in {\mathbb R}^d$.
\end{enumerate}
\end{lemma}

\subsubsection{Covering of $\omegav$, $\omegae$, and $\omegaf$}

We start with coverings of vertex, edge and face neighborhoods and provide coverings using balls 
insribed in $\Omega$ whose size is proportional 
to their distance to the vertex, edge or face, respectively.
\begin{lemma}[covering of $\omega_{\bullet}$, $\bullet\in\{\mathbf{v}, \mathbf{e}, \mathbf{f}\}$]
\label{lemma:covering-omega_c}
Given ${\bullet} \in \mathcal{V}  \cup \mathcal{E} \cup \mathcal{F} $ and $\omegaeps > 0$, 
there are parameters $0 < c < \widehat c < 1$ as well as 
points $(x_i)_{i \in {\mathbb N}}\subset \omega_{\bullet} = \omega^\omegaeps_{\bullet}$ 
such that: 
\begin{enumerate}[(i)]
\item 
The collection 
${\mathcal B}\coloneqq  \{B_i\coloneqq  B_{c \operatorname{dist}(x_i,{\bullet})}(x_i) \,|\, i \in {\mathbb N}\}$ 
of open balls covers $\omega_{\bullet}$.  
\item
The collection 
$\widehat {\mathcal B}\coloneqq  \{\widehat B_i\coloneqq  B_{\widehat c \operatorname{dist}(x_i,{\bullet})}(x_i) \,|\, i \in {\mathbb N}\}$ 
of open balls satisfies a finite overlap property, i.e., there is an integer $N > 0$ 
depending only on the spatial dimension $d = 3$ and the parameters $c$, $\hc$ such that 
$\card \{i\,|\, x \in \widehat B_i\} \leq N$ for all $x \in \R^3$. 
The balls from $\widehat{\mathcal B}$ are contained in $\Omega$. 
\end{enumerate}
\end{lemma}
\begin{proof}
Apply Lemma~\ref{lemma:MW} with $M = \{\bullet\}$ 
and sufficiently small parameters $c$, $\zeta>0$. 
Observe that by possibly slightly increasing the parameter $c$, 
one can ensure that the open balls rather than the closed balls 
given by Lemma~\ref{lemma:MW} cover $\omega_{\bullet}$. 
Also, since $c < 1$, 
the index set ${\mathcal I}$ of Lemma~\ref{lemma:MW} 
cannot be finite so that we may assume 
${\mathcal I} = {\mathbb N}$.  
%
\end{proof}
\subsubsection{Covering of $\omegaef$}
We now introduce a covering of edge-face neighborhoods $\omegaef$. 
We start by a
covering of half-balls resting on the face $\bff$ and with 
size proportional to the distance from the edge.
\begin{lemma}
\label{lemma:covering-omega_ef}
Given $\bfe \in \cE$, $\bff \in \cF_\bfe$, there are $\omegaeps> 0$ and parameters $0 < c< \widehat c < 1$ 
as well as points $(x_i)_{i \in {\mathbb N}} \subset \bff$ such that,
denoting $R_i = c\dist(x_i, \bfe)$ and $\hR_i = \hc \dist(x_i, \bfe)$:
\begin{enumerate}[(i)]
\item 
The sets $H_i\coloneqq  B_{R_i}(x_i) \cap \Omega$ are half-balls and the collection 
$\cB\coloneqq  \{H_i\,|\, i \in {\mathbb N}\}$ covers $\omegaef = \omegaef^\omegaeps$.  

\item
The collection $\hcB\coloneqq  \{\widehat H_i\coloneqq  B_{\hR_i}(x_i) \cap \Omega\}$ is a collection of half-balls and 
satisfies a finite overlap property, i.e., there is $N > 0$ 
depending only on the spatial dimension $d = 3$ and the parameters $c$, $\hc$ 
such that 
$\card \{i\,|\, x \in \widehat H_i\} \leq N$ for all $x \in \R^3$. 
\end{enumerate}
\end{lemma}
\begin{proof}
Let $\widetilde{\bff}$ be the (infinite) plane containing $\bff$. 
We apply Lemma~\ref{lemma:MW} to the 2D plane surface $\bff\cap \partial\omegaef^\omegaeps$ (for some sufficiently small $\omegaeps$) 
 and $M\coloneqq \{\bfe \}$ and the parameter $c$ sufficiently small so that 
$B_{2 c \dist(x,\bfe)}(x) \cap \Omega $ is a half-ball for all $x \in \bff \cap \partial\omegaef^\omegaeps$. 
Lemma~\ref{lemma:MW} provides a collection $(x_i)_{i \in {\mathbb N}} \subset
\bff$ such that the balls 
$B_i\coloneqq  B_{R_i}(x_i) \subset {\mathbb R}^3$
and the scaled balls $\widehat{B}_i\coloneqq  B_{c (1+\zeta)\dist(x_i,\bfe)}(x_i) \subset {\mathbb R}^3$ (for suitable, sufficiently small $\zeta$)
satisfy the following: the 2D balls $\{B_i\cap \widetilde{\bff}\,|\, i \in {\mathbb N}\}$  
cover $\partial\omegaef^\omegaeps \cap \bff$, and the 2D balls 
$\{\widehat B_i\cap \widetilde {\bff}\,|\, i \in {\mathbb N}\}$ satisfy a finite overlap condition on $\widetilde{\bff}$. 
By possibly slightly increasing the parameter $c$ (e.g., by replacing $c$ with $c(1+\zeta/2)$), 
the newly defined balls $B_i$ then cover a set 
$\omegaef^{\omegaeps}$ 
for a possibly reduced $\omegaeps$. 
It remains to see that the balls $\widehat B_i$ satisfy a finite overlap condition on ${\mathbb R}^2$: 
given $x \in \widehat B_i$,
its projection $x_{\bff}$ onto $\widetilde{\bff}$ satisfies $x_{\bff} \in \widehat B_i\cap \widetilde{\bff}$ 
since $x_i \in {\bff} \subset \widetilde{\bff}$. 
This 
implies that the overlap constants of the 3D balls $\widehat B_i$ in ${\mathbb
  R}^3$ is the same as the overlap constant of the 2D balls
$\widehat B_i \cap \widetilde{\bff}$ in $\widetilde{\bff}$. 
The
half-balls $H_i\coloneqq B_i \cap \Omega$ and $\widehat H_i\coloneqq  \widehat B_i \cap \Omega$ 
have the stated properties. 
%
\end{proof}
\subsubsection{Covering of $\omegavf$}
Similarly, we provide a covering of the vertex-face neighborhoods $\omegavf$ using half-balls centered on the face $f$.
\begin{lemma}
  \label{lemma:covering-omega_vf}
  Given $\bfv \in \cV$, $\bff \in \cF_\bfv$, there are $\omegaeps> 0$ and parameters $0 < c< \widehat c < 1$ 
  as well as points $(x_i)_{i \in {\mathbb N}} \subset \bff$ such that,
  denoting $R_i = c\dist(x_i, \bfv)$ and $\hR_i = \hc \dist(x_i, \bfv)$:
  \begin{enumerate}[(i)]
  \item 
    The sets $H_i\coloneqq  B_{R_i}(x_i) \cap \Omega$ are half-balls and the collection 
    $\cB\coloneqq  \{H_i\,|\, i \in {\mathbb N}\}$ covers $\omegavf = \omegavf^\omegaeps$.  
  \item
    The collection $\hcB\coloneqq  \{\widehat H_i\coloneqq  B_{\hR_i}(x_i) \cap \Omega\}$ is a collection of half-balls and 
    satisfies a finite overlap property, i.e., 
there is $N > 0$ depending only on the spatial dimension $d = 3$ and the parameters $c$, $\hc$ 
such that 
    $\card \{i\,|\, x \in \widehat H_i\} \leq N$ for all $x \in \R^3$. 
  \end{enumerate}
\end{lemma}
\begin{proof}
  The proof is the same as the proof of Lemma \ref{lemma:covering-omega_ef}.
\end{proof}

 
\subsubsection{Covering of $\omegave$}
For the vertex-edge neighborhoods $\omegave$, 
we introduce a covering using wedges centered on the edge with size proportional to the distance to the vertex.
\begin{lemma}
\label{lemma:covering-omega_ve}
Given $\bfv \in \cV$, $\bfe \in \cE_\bfv$, there are $\omegaeps> 0$ and parameters $0 < c< \widehat c < 1$ 
as well as points $(x_i)_{i \in {\mathbb N}} \subset \bfe$ such that,
denoting $R_i = c\dist(x_i, \bfv)$ and $\hR_i = \hc \dist(x_i, \bfv)$:
\begin{enumerate}[(i)]
\item 
The collection of wedges 
$\cB\coloneqq  \{W_i\subset  B_{R_i}(x_i) \cap \Omega\}_{i\in\N}$ covers $\omegave = \omegave^\omegaeps$.  

\item
The collection of wedges $\hcB\coloneqq  \{\widehat W_i\subset  B_{\hR_i}(x_i) \cap
\Omega\}_{i\in\N}$ satisfies $W_i\subset \widehat W_i$ 
and a finite overlap property, i.e., 
there is $N > 0$ depending only on the spatial dimension $d = 3$ and the parameters $c$, $\hc$ 
such that 
    $\card \{i\,|\, x \in \widehat W_i\} \leq N$ for all $x \in \R^3$.
%
\end{enumerate}
\end{lemma}
\begin{proof}
Let $\widetilde{\bfe}$ be the (infinite) line containing $\bfe$. 
We apply Lemma~\ref{lemma:MW} to the intervals $\bfe\cap \partial\omegave^\omegaeps$ (for some sufficiently small $\omegaeps$) 
 and $M\coloneqq \{\bfv \}$ and the parameter $c$ sufficiently small so that 
$B_{2 c \dist(x,\bfe)}(x) \cap \Omega $ is a wedge for all $x \in \bfe \cap \partial\omegave^\omegaeps$. 
Lemma~\ref{lemma:MW} provides a collection $(x_i)_{i \in {\mathbb N}} \subset
\bfe$ such that the balls 
$B_i\coloneqq  B_{R_i}(x_i) \subset {\mathbb R}^3$
and the scaled balls $\widehat{B}_i\coloneqq  B_{c (1+\zeta)\dist(x_i,\bfv)}(x_i) \subset {\mathbb R}^3$ (for suitable, sufficiently small $\zeta$)
satisfy the following: the intervals $\{B_i\cap \widetilde{\bfe}\,|\, i \in {\mathbb N}\}$  
cover $\partial\omegave^\omegaeps \cap \bfe$, and the intervals
$\{\widehat B_i\cap \widetilde {\bfe}\,|\, i \in {\mathbb N}\}$ satisfy a finite overlap condition on $\widetilde{\bfe}$. 
Upon increasing the parameter $c$ (e.g., by replacing $c$ with $c(1+\zeta/2)$), 
the newly defined balls $B_i$ then cover a set 
$\omegave^{\omegaeps}$ 
for a possibly reduced $\omegaeps$. 
It remains to see that the balls $\widehat B_i$ satisfy a finite overlap condition on ${\mathbb R}^2$: 
given $x \in \widehat B_i$,
its projection $x_{\bfe}$ onto $\widetilde{\bfe}$ satisfies $x_{\bfe} \in \widehat B_i\cap \widetilde{\bfe}$ 
since $x_i \in {\bfe} \subset \widetilde{\bfe}$. 
This 
implies that the overlap constants of the balls $\widehat B_i$ in ${\mathbb
  R}^3$ is the same as the overlap constant of the intervals
$\widehat B_i \cap \widetilde{\bfe}$ in $\widetilde{\bfe}$. 
The
wedges $W_i\coloneqq B_i \cap \Omega$ and $\widehat W_i\coloneqq  \widehat B_i \cap \Omega$ 
have the stated properties. 
%
\end{proof}

\subsubsection{Covering of $\omegavef$}
\label{sec:CovVEF}
In the same way, we obtain a covering of the vertex-edge-face neighborhoods $\omegavef$. 
\begin{lemma}
  \label{lemma:covering-omega_vef}
  Given $\bfv \in \cV$, $\bfe \in \cE_\bfv$, and $\bff\in \cF_\bfe \cap
  \cF_\bfv$, there are $\omegaeps> 0$ and parameters $0 < c< \widehat c < 1$ 
  as well as points $(x_i)_{i \in {\mathbb N}} \subset \bfe$ such that,
  denoting $R_i = c\dist(x_i, \bfv)$ and $\hR_i = \hc \dist(x_i, \bfv)$:
  \begin{enumerate}[(i)]
  \item 
    The sets $W_i\coloneqq  B_{R_i}(x_i) \cap \Omega$ are wedges  and the collection 
    $\cB\coloneqq  \{W_i\,|\, i \in {\mathbb N}\}$ covers $\omegavef = \omegavef^\omegaeps$.  

  \item
  The collection $\hcB\coloneqq  \{\widehat W_i\coloneqq  B_{\hR_i}(x_i) \cap \Omega\}$ 
  is a collection of wedges and 
  satisfies a finite overlap property, i.e., there is $N > 0$ 
depending only on the spatial dimension $d = 3$ and the parameters $c$, $\hc$ such that 
    $\card \{i\,|\, x \in \widehat W_i\} \leq N$ for all $x \in \R^3$. 
  \end{enumerate}
\end{lemma}
\begin{proof}
  The proof is the same as that of Lemma \ref{lemma:covering-omega_ve}, with
  $\omegavef$ replacing $\omegave$.
\end{proof}

\subsection{Weighted $H^p$-regularity for the CS extension}
\label{sec:WgtHpRegExtPbm}
In the following, 
we provide separate weighted analytic regularity estimates
on extensions of each neighborhood $\omega_{\bullet}$ 
used to decompose $\Omega$ in \eqref{eq:Nghbrhoods}.
Hereby, for any set $\omega \subset \R^3$ and $\boundY>0$, 
define
$\omega^\boundY \coloneqq \omega \times (0, \boundY)$.
\subsubsection{Vertex neighborhoods $\omegav$}
We have 
\begin{align*}
r_{\mathbf{f}} \sim r_{\mathbf{e}} \sim r_{\mathbf{v}} \qquad \text{ on $\omegav$}. 
\end{align*}

The following lemma provides higher order regularity estimates in  
vertex-weighted norms for solutions 
to the Caffarelli-Silvestre extension problem with smooth data.

\begin{lemma}[Weighted $H^p$-regularity in $\omegav$] 
\label{lemma:regularity-omega_c}
Let $\omegav = \omegav^\omegaeps$ be given for some $\omegaeps > 0$ and $\bfv\in\cV$.  Let $U$ be the solution of \eqref{eq:minimization}.
There is $\gamma > 0$ depending only on $s$, $\Omega$, $\omegav$, and $\boundY$, 
and for every $\varepsilon \in (0,1/2)$, there exists $C_\varepsilon>0$ 
depending additionally on $\varepsilon$
such that for all $\beta\in \N^3_0$, with $p = \betam$,
\begin{align*}
  \|\rv^{p-1/2+\varepsilon} \dbeta \nabla U\|_{L^2_{\alpha}(\omegav\times (0, \boundY))}^2 &\leq
C_\varepsilon \gamma^{2p+1}p^{2p} \bigg[ \|f\|^2_{H^{1}(\Omega)} + \|F\|^2_{L^2_{-\alpha}(\R^3\times(0,\boundY))}   
       \\ & \qquad +
       \sum_{j=1}^{p} p^{-2j}\left(\max_{\etam=j}\norm{\deta f}^2_{L^2(\Omega)} 
  +\max_{\etam=j-1} \norm{\deta F}^2_{L^2_{-\alpha}(\R^3\times(0,\boundY))} \right)\bigg].
\end{align*}
\end{lemma}
\begin{proof} 
The case $p = 0$ follows from Lemma~\ref{lem:estH13D} and the 
estimates \eqref{eq:CUFf}, \eqref{eq:CUFf-simplified}. 

We therefore assume in the remainder of this proof that 
$p \in \N$.
Lemma~\ref{lemma:covering-omega_c} gives
the covering $ \bigcup_i B_i \supset \omegav$
with 
scaled balls $B_i = B_{c \rv(x_i)}(x_i)$ and 
scaled balls $\hB_i = B_{\hc \rv(x_i)}(x_i)$.
We denote $R_i\coloneqq  \hc \dist(x_i,\bfv)$
the radius of the ball $\hB_i$ and note that, for some $C_B > 1$, 
\begin{equation}
\label{eq:estimate-r_c}
\forall i \in \N\quad \forall x \in \hB_i 
\qquad 
C^{-1}_B R_i \leq \rv(x) \leq C_B  R_i. 
\end{equation}
We assume (for convenience) that $R_i \leq 1$ for all $i$. 

For any multi index $\beta$, with $p = \betam$,
\begin{align*}
  \|\rv^{p-1/2+\varepsilon} \dbeta \nabla U\|_{L^2_{\alpha}(\omegav^{\boundY/2})}^2 
  \overset{\text{L.~\ref{lemma:covering-omega_c}}}
  &{\leq}
    \sum_{i\in\N} \|\rv^{p-1/2+\varepsilon} \dbeta \nabla U\|_{L^2_{\alpha}(B_i^{\boundY/2})}^2 
  \\
  \overset{\eqref{eq:estimate-r_c}}
  &{\leq}
    \sum_{i\in\N} (C_BR_i)^{2p+\varepsilon}\|\rv^{-1/2+\varepsilon/2}\dbeta \nabla U\|_{L^2_{\alpha}(B_i^{\boundY/2})}^2 
  \\
  \overset{\text{C.~\ref{lem:localhighregularity3D-int}}}
  &{\lesssim}
    \begin{multlined}[t][0.6\linewidth]
      \sum_{i\in\N} (C_BR_i)^{2p+\varepsilon}(\gamma_1 p)^{2p+1}R_i^{-2p-1}
    \bigg[
    \|\nabla U\|_{L^2_{\alpha}(\hB_i^\boundY)}^2
    + R_i^{s+1}\tNp_{\hB_i^\boundY}(F, f)
    \bigg]
  \end{multlined}
  \\
  & \leq
    C_B^{2p}(\gamma_1p)^{2p+1}\sum_{i\in\N}\bigg[C_B\|\rv^{-1/2+\varepsilon/2} \nabla U\|_{L^2_\alpha(\hB_i^\boundY)}^2 + R_i^{s+\varepsilon}\tNp_{\hB_i^\boundY}(F, f)\bigg]
    \\ & \lesssim 
         C_B^{2p}(\gamma_1p)^{2p+1}\bigg[C_B\|\rv^{-1/2+\varepsilon/2} \nabla U\|_{L^2_\alpha(\omegav^{\hat{\omegaeps}}\times(0, \boundY))}^2 + \tNp_{\Omega^+}(F, f)\bigg].
\end{align*}
We conclude by using that in $\omegav$, $\rv\simeq r_{\partial \Omega}$ and
using Lemma~\ref{lem:estH13D}, Lemma \ref{lem:regularity3D} and \eqref{eq:CUFf-simplified}.
\end{proof}

\subsubsection{Edge-neighborhoods $\omegae$}
%
We have 
\begin{align*}
r_{\mathbf{f}} \sim r_{\mathbf{e}} \qquad \text{ on $\omegae$}. 
\end{align*}
We start with a weighted regularity estimate on arbitrary wedges centered on an edge $\bfe$.

\begin{lemma}[Weighted $H^p$-regularity in a wedge] 
\label{lemma:regularity-W}
Let $\bfe\in\cE$, $x_0\in \bfe$, $R>0$, $\zeta>0$ and let
\begin{equation*}
  W_R = B_R(x_0) \cap \{x\in \Omega: \rhoef(x) > \zeta\,\, \forall \bff\in \cF_{\bfe}\}
\end{equation*}
be a wedge either in $\omega_{\bfe}$ or $\omega_{\bfv\bfe}$. 
Let $c\in(0,1)$ and
let $U$ be the solution of \eqref{eq:minimization}.

Then,
there exists $\gamma > 0$ depending only on $s$, $\Omega$, $\zeta$
and $\boundY$, 
and for every $\varepsilon \in (0,1/2)$, there exists $C_\varepsilon>0$ 
depending additionally on $\varepsilon$
such that for all $\betaperp = (\beta_{\perp,1}, \beta_{\perp,2})\in \N^2_0$ and
all $ \beta_{\parallel}\in \N_0$,
with $\pperp = \beta_{\perp, 1} + \beta_{\perp, 2}$, $\ppar = \beta_\parallel$, 
and $p=\pperp+\ppar$,
it holds that
\begin{multline}
  \label{eq:estimate-omegave}
  \|\re^{\pperp-1/2+\varepsilon} \Dperpe^{\betaperp}\Dpare^{\beta_{\parallel}} \nabla U\|_{L^2_{\alpha}(W_{cR}^{\boundY/4})}^2
   \leq
    C_\varepsilon \gamma^{2p+1}p^{2p} \bigg[R^{-2\ppar-1}\bigg(\|
    \nabla U\|_{L^2_{\alpha}(W_R^{\boundY})}^2
\\    +R^{s+1}\tNppar_{W_R^\boundY}(F, f)\bigg)
    +\tNpperp_{W_R^{\boundY/2}}(\Dpare^{\ppar}F, \Dpare^{\ppar}f) \bigg]
\end{multline}
where $\Dperpe^{\betaperp} = \Doneperpe^{\beta_{\perp,
    1}}\Dtwoperpe^{\beta_{\perp, 2}}$.
\end{lemma}
\begin{proof} 
The case $\pperp = 0$ follows from Lemma~\ref{lem:localhighregularity3D-ef} 
and from the estimates \eqref{eq:CUFf}, \eqref{eq:CUFf-simplified}. 

We therefore assume in the following that
$\pperp \in \N$.
Denote $\tc = (c+1)/2\in(c, 1)$. 

We observe that the argument of 
Lemma~\ref{lemma:covering-omega_c} also gives
a covering $ \bigcup_i B_i \supset W_{cR}$
with balls $B_i = B_{c_1 \re(x_i)}(x_i)$ and scaled 
balls $\hB_i = B_{\hc_1 \re(x_i)}(x_i)$ such that 
$\bigcup_i \hB_i \subset W_{\tc R}$, provided 
one chooses the parameters $c_1,\hc_1 > 1$ small enough.

We denote $R_i\coloneqq  \hc_1 \dist(x_i,\bfe)$
the radius of the ball $\hB_i$ and note that, for some $C_B > 1$, 
\begin{equation}
\label{eq:estimate-r_e}
\forall i \in \N\quad \forall x \in \hB_i 
\qquad 
C^{-1}_B R_i \leq \re(x) \simeq r_{\partial\Omega}(x) \leq C_B  R_i. 
\end{equation}
We assume (for convenience) that $R_i \leq 1$ for all $i$. 

We apply Lemma~\ref{lem:localhighregularity3D-int} to the function 
$\Dpare^{\ppar} U$ (noting that this function satisfies \eqref{eq:extension3D} with data $\Dpare^{\ppar}f$, $\Dpare^{\ppar}F$)
with the pair ($B_i$, $\hB_i$) of
concentric balls, with $\boundY/2$ instead of $\boundY$,
and with constant denoted $\gamma_1\geq 1$.
For any 
$\betaperp = (\beta_{\perp,1}, 
\beta_{\perp,2})\in \N^2_0$ and
$ \beta_{\parallel}\in \N_0$, 
with 
$\pperp = |\betaperp|\in\N$ and $\ppar = \beta_\parallel$, 
it holds that
\begin{multline*}
  \|\re^{\pperp-1/2+\varepsilon} \Dperpe^{\betaperp}\Dpare^{\ppar} \nabla U\|_{L^2_{\alpha}(W_{cR}^{\boundY/4})}^2
  \\
\begin{aligned}
  \overset{\text{L.~\ref{lemma:covering-omega_c}}}
  &{\leq}
    \sum_{i\in\N} \|\re^{\pperp-1/2+\varepsilon} \Dperpe^{\betaperp}\Dpare^{\ppar} \nabla U\|_{L^2_{\alpha}(B_i^{\boundY/4})}^2
  \\
  \overset{\eqref{eq:estimate-r_e}}
  &{\leq}
    \sum_{i\in\N} (C_BR_i)^{2\pperp+\varepsilon}\| \re^{-1/2+\varepsilon/2}\Dperpe^{\betaperp}\Dpare^{\ppar}\nabla U\|_{L^2_{\alpha}(B_i^{\boundY/4})}^2 
  \\
  \overset{\text{L.~\ref{lem:localhighregularity3D-int}}}
  &{\leq}
    \begin{multlined}[t][0.6\linewidth]
    \sum_{i\in\N} (C_BR_i)^{2\pperp+\varepsilon}
    (\gamma_1 {\pperp})^{2\pperp+1}R_i^{-2\pperp-1}\bigg[
    \|\Dpare^{\ppar}\nabla U\|_{L^2_{\alpha}(\hB_i^{\boundY/2})}^2 +
    R_i^{s+1}\tNpperp_{\hB_i^{\boundY/2}}(\Dpare^{\ppar}F, \Dpare^{\ppar}f) \bigg]
  \end{multlined}
  \\
  \overset{\eqref{eq:estimate-r_e}}
  &{\lesssim}
    C_B^{{2\pperp+1}}(\gamma_1 {\pperp})^{2\pperp+1}\sum_{i\in\N}\bigg[\|
    \re^{-1/2+\varepsilon/2}\Dpare^{\ppar}\nabla
    U\|_{L^2_{\alpha}(\hB_i^{\boundY/2})}^2
    +R^{s+\varepsilon}\tNpperp_{\hB_i^{\boundY/2}}(\Dpare^{\ppar}F, \Dpare^{\ppar}f) \bigg]
  \\
  &\lesssim
    C_B^{{2\pperp+1}}(\gamma_1 {\pperp})^{2\pperp+1}\bigg[\|
    \re^{-1/2+\varepsilon/2}\Dpare^{\ppar}\nabla
    U\|_{L^2_{\alpha}(W_{\tc R}^{\boundY/2})}^2
    +\tNpperp_{W_{\tc R}^{\boundY/2}}(\Dpare^{\ppar}F, \Dpare^{\ppar}f) \bigg]
  \\
  \overset{\text{L.~\ref{lem:localhighregularity3D-ef}}}
  &{\leq}
    \begin{multlined}[t][\arraycolsep]
    C_B^{{2\pperp+1}}(\gamma_1 {\pperp})^{2\pperp+1}(\gamma_2 {\ppar})^{2\ppar+1}\bigg[R^{-2\ppar-1}\bigg(\|
    \nabla U\|_{L^2_{\alpha}(W_R^{\boundY})}^2
    +R^{s+1}\tNppar_{W_R^\boundY}(F, f)\bigg)
\\    +\tNpperp_{W_R^{\boundY/2}}(\Dpare^{\ppar}F, \Dpare^{\ppar}f) \bigg],
  \end{multlined}
\end{aligned}
\end{multline*}
where we have used Lemma~\ref{lem:localhighregularity3D-ef} in the last step. 
\end{proof}
\begin{corollary}
Let $\bfe\in\cE$ and $\boundY>0$. 
Let $U$ be the solution of \eqref{eq:minimization}.

Then,
there exists $\gamma > 0$ depending only on $s$, $\Omega$, $\zeta$ and $\boundY$, 
and, for every $\varepsilon \in (0,1/2)$, there exists $C_\varepsilon>0$ 
depending additionally on $\varepsilon$
such that for all $\betaperp = (\beta_{\perp,1}, \beta_{\perp,2})\in \N^2_0$ and
all $ \beta_{\parallel}\in \N_0$, 
with $\pperp = \beta_{\perp, 1} + \beta_{\perp, 2}$, $\ppar = \beta_\parallel$, and $p=\pperp+\ppar$, 
it holds that
  \begin{equation}
      \|\re^{\pperp-1/2+\varepsilon} \Dperpe^{\betaperp}\Dpare^{\beta_{\parallel}} \nabla U\|_{L^2_{\alpha}(\omegae^{\boundY/4})}^2
      \leq
        C_\varepsilon \gamma^{2p+1}p^{2p}   \tNp_{\Omega^\boundY}(F, f).
  \end{equation}
\end{corollary}
\begin{proof}
  This follows directly from Lemma~\ref{lemma:regularity-W} with $R\simeq 1$ and from \eqref{eq:CUFf-simplified}.
\end{proof}
\subsubsection{Vertex-edge neighborhoods $\omegave$}
We have 
\begin{align*}
r_{\mathbf{f}} \sim r_{\mathbf{e}} \qquad \text{ and } \qquad r_{\mathbf{e}} \leq r_{\mathbf{v}} \qquad \text{ on $\omegave$}. 
\end{align*}
\begin{lemma}[Weighted $H^p$-regularity in $\omegave$] 
\label{lemma:regularity-omegave}
Let $U$ be the solution of \eqref{eq:minimization}.
There is $\gamma > 0$ depending only on $s$, $\Omega$, and $\boundY$, 
and for every $\varepsilon \in (0,1/2)$, there exists $C_\varepsilon>0$ 
depending additionally on $\varepsilon$
such that for all 
$\betaperp = (\beta_{\perp,1}, \beta_{\perp,2})\in \N^2_0$ and
$ \beta_{\parallel}\in \N_0$,
with $\pperp = \beta_{\perp, 1} + \beta_{\perp, 2}$, $\ppar = \beta_\parallel$, and $p=\pperp+\ppar$, 
it holds that
\begin{equation}
  \label{eq:estimate-W}
\begin{aligned}
  \|\rv^{\ppar+\varepsilon}\re^{\pperp-1/2+\varepsilon} \Dperpe^{\betaperp}\Dpare^{\beta_{\parallel}} \nabla U\|_{L^2_{\alpha}(\omegave^{\boundY/4})}^2
  &\leq
C_\varepsilon \gamma^{2p+1}p^{2p}  \tNp_{\Omega^\boundY}(F, f),
\end{aligned}
\end{equation}
where $\Dperpe^{\betaperp} = \Doneperpe^{\beta_{\perp,
    1}}\Dtwoperpe^{\beta_{\perp, 2}}$.
\end{lemma}
\begin{proof}
We use the covering of wedges $W_i \subset B_{cR_i}(x_i)$  
with $\hW_i \subset B_{R_i}(x_i)$ given by Lemma~\ref{lemma:covering-omega_ve}. 
We have, for a constant $C_W > 1$, 
  \begin{equation*}
      \forall i \in \N\quad \forall x \in \hW_i 
      \qquad 
      C^{-1}_W R_i \leq \rv(x)  \leq C_W  R_i. 
  \end{equation*}
  Using this and Lemma~\ref{lemma:regularity-W},
  \begin{multline*}
    \|\rv^{\ppar+\varepsilon}\re^{\pperp-1/2+\varepsilon} \Dperpe^{\betaperp}\Dpare^{\beta_{\parallel}} \nabla U\|_{L^2_{\alpha}(\omegave^{\boundY/4})}^2 
    \\
    \begin{aligned}[t]
      &\leq \sum_{i\in \N} (C_WR_i)^{2\ppar+2\varepsilon} \|\re^{\pperp-1/2+\varepsilon} \Dperpe^{\betaperp}\Dpare^{\beta_{\parallel}} \nabla U\|_{L^2_{\alpha}(W_i^{\boundY/4})}^2
       \\
      &
        \begin{multlined}[t][\arraycolsep]
        \leq \sum_{i\in \N} (C_WR_i)^{2\ppar+2\varepsilon}  \gamma_1^{2p+1}p^{2p} \bigg[R_i^{-2\ppar }
        \bigg(R_i^{-1}\|\nabla U\|_{L^2_\alpha(\hW_i^\boundY)}^2 +
              R_i^{s} \tNppar_{\hW_i^\boundY}(F, f)\bigg)  
        \\ 
        + \tNpperp_{\hW_i^{\boundY/2}}(\Dpare^{\ppar}F, \Dpare^{\ppar}f)\bigg]
      \end{multlined}
        \\
      &\leq  \gamma^{2p+1}p^{2p}\sum_{i\in \N} \bigg[ \|\rv^{-1/2+\varepsilon}\nabla U\|_{L^2_\alpha(\hW^\boundY_i)}^2 +
        \tNppar_{\hW_i^\boundY}(F, f) + \tNpperp_{\hW_i^{\boundY/2}}(\Dpare^{\ppar}F, \Dpare^{\ppar}f)\bigg].
    \end{aligned}
  \end{multline*}
  The bound $ \rv(x)\geq r_{\partial \Omega}(x)$, 
  the finite overlap of the wedges $\hW_i$, 
  Lemma \ref{lem:estH13D}, and \eqref{eq:CUFf-simplified} conclude the proof.
\end{proof}

\subsubsection{Face neighborhoods $\omegaf$}
\label{sec:omegaf}
We write $H_R^\boundY \coloneqq H_R \times (0, \boundY)$ 
and start with a weighted regularity estimate on arbitrary half-balls centered on a face $\bff$.
\begin{lemma}[Weighted $H^p$-regularity in a half-ball] 
\label{lemma:regularity-halfball}
Let $\bff\in\cF$, $x_0\in \bff$, $R>0$, $\zeta>0$ and let
\begin{equation*}
  H_R = B_R(x_0) \cap \Omega
\end{equation*}
be a half-ball. Let $c\in(0,1)$ and
let $U$ be the solution of \eqref{eq:minimization}.
There is $\gamma > 0$ depending only on $s$, $\Omega$, $\zeta$
and $\boundY$, 
and for every $\varepsilon \in (0,1/2)$, there exists $C_\varepsilon>0$ 
depending additionally on $\varepsilon$
such that for all $\betapar = (\beta_{\parallel,1}, \beta_{\parallel,2})\in \N^2_0$ and
$ \betaperp\in \N_0$,
with $\ppar = \beta_{\parallel, 1} + \beta_{\parallel, 2}$, $\pperp = \betaperp$, 
and $p=\ppar+\pperp$,
it holds that
\begin{multline}
  \label{eq:estimate-halfball}
  \|\rf^{\pperp-1/2+\varepsilon} \Dperpf^{\betaperp}\Dparf^{\beta_{\parallel}} \nabla U\|_{L^2_{\alpha}(H_{cR}^{\boundY/4})}^2
   \leq
    C_\varepsilon \gamma^{2p+1}p^{2p} \bigg[R^{-2\ppar-1}\bigg(\|
    \nabla U\|_{L^2_{\alpha}(H_R^{\boundY})}^2
  \\  +R^{s+1}\tNppar_{W_R^\boundY}(F, f)\bigg)
    +\tNpperp_{H_R^{\boundY/2}}(\Dparf^{\ppar}F, \Dparf^{\ppar}f) \bigg]
\end{multline}
where $\Dparf^{\betapar} = \Doneparf^{\beta_{\parallel, 1}}\Dtwoparf^{\beta_{\parallel, 2}}$.
\end{lemma}
\begin{proof} 
The case $\pperp = 0$ follows from Lemma~\ref{lem:localhighregularity3D-ef} and the estimates \eqref{eq:CUFf}, \eqref{eq:CUFf-simplified}. 
We therefore assume $\pperp \in \N$.

Denote $\tc = (c+1)/2\in(c, 1)$.
The arguments of Lemma~\ref{lemma:covering-omega_c} give
a covering $ \bigcup_i B_i \supset H_{cR}$
with balls $B_i = B_{c_1 \rf(x_i)}(x_i)$ and 
scaled balls $\hB_i = B_{\hc_1 \rf(x_i)}(x_i)$ 
such that $\bigcup_i \hB_i \subset H_{\tc R}$, 
if one chooses the parameters $c_1,\hc_1 > 1$ small enough.

We denote $R_i\coloneqq  \hc_1 \dist(x_i,\bff)$ the radius of the ball $\hB_i$ and note that, for some $C_B > 1$, 
\begin{equation}
\label{eq:estimate-r_f}
\forall i \in \N\quad \forall x \in \hB_i 
\qquad 
C^{-1}_B R_i \leq \rf(x) = r_{\partial\Omega}(x) \leq C_B  R_i. 
\end{equation}
We assume (for convenience) that $R_i \leq 1$ for all $i$. 

We apply Lemma~\ref{lem:localhighregularity3D-int} to the function 
$\Dparf^{\betapar} U$ (noting that this function satisfies \eqref{eq:extension3D} with data $\Dparf^{\betapar}f$, $\Dparf^{\betapar}F$)
with the pair ($B_i$, $\hB_i$) of
concentric balls, with $\boundY/2$ instead of $\boundY$,
and with constant denoted $\gamma_1\geq 1$.
For any 
$\betapar = (\beta_{\parallel,1}, \beta_{\parallel,2})\in \N^2_0$ and
$ \betaperp\in \N_0$, with $\ppar = |\betapar|\in\N$ and $\pperp = \betaperp$, 
it holds that
\begin{multline*}
  \|\rf^{\pperp-1/2+\varepsilon} \Dperpf^{\betaperp}\Dparf^{\ppar} \nabla U\|_{L^2_{\alpha}(H_{cR}^{\boundY/4})}^2
  \\
\begin{aligned}
  \overset{\text{L.~\ref{lemma:covering-omega_c}}}
  &{\leq}
    \sum_{i\in\N} \|\rf^{\pperp-1/2+\varepsilon} \Dperpf^{\betaperp}\Dparf^{\betapar} \nabla U\|_{L^2_{\alpha}(B_i^{\boundY/4})}^2
  \\
  \overset{\eqref{eq:estimate-r_f}}
  &{\leq}
    \sum_{i\in\N} (C_BR_i)^{2\pperp+\varepsilon}\| \rf^{-1/2+\varepsilon/2}\Dperpf^{\betaperp}\Dparf^{\betapar}\nabla U\|_{L^2_{\alpha}(B_i^{\boundY/4})}^2 
  \\
  \overset{\text{L.~\ref{lem:localhighregularity3D-int}}}
  &{\leq}
    \begin{multlined}[t][\arraycolsep]
    \sum_{i\in\N} (C_BR_i)^{2\pperp+\varepsilon}
    (\gamma_1 {\pperp})^{2\pperp+1}R_i^{-2\pperp-1}\bigg[
    \|\Dparf^{\betapar}\nabla U\|_{L^2_{\alpha}(\hB_i^{\boundY/2})}^2 \\ +
    R_i^{s+1}\tNpperp_{\hB_i^{\boundY/2}}(\Dparf^{\betapar}F, \Dparf^{\betapar}f) \bigg]
  \end{multlined}
  \\
  \overset{\eqref{eq:estimate-r_e}}
  &{\lesssim}
    C_B^{2\pperp}(\gamma_1 {\pperp})^{2\pperp+1}\sum_{i\in\N}\bigg[\|
    \rf^{-1/2+\varepsilon/2}\Dparf^{\betapar}\nabla
    U\|_{L^2_{\alpha}(\hB_i^{\boundY/2})}^2
    +R^{s+\varepsilon}\tNpperp_{\hB_i^{\boundY/2}}(\Dparf^{\betapar}F, \Dparf^{\betapar}f) \bigg]
  \\
  &\lesssim
    C_B^{2\pperp}(\gamma_1 {\pperp})^{2\pperp+1}\bigg[\|
    \rf^{-1/2+\varepsilon/2}\Dparf^{\betapar}\nabla
    U\|_{L^2_{\alpha}(H_{\tc R}^{\boundY/2})}^2
    +\tNpperp_{H_{\tc R}^{\boundY/2}}(\Dparf^{\betapar}F, \Dparf^{\betapar}f) \bigg]
  \\
  \overset{\text{L.~\ref{lem:localhighregularity3D-ef}}}
  &{\leq}
    \begin{multlined}[t][\arraycolsep]
    C_B^{2\pperp}(\gamma_1 {\pperp})^{2\pperp+1}(\gamma_2 {\ppar})^{2\ppar+1}\bigg[R^{-2\ppar-1}\bigg(\|
    \nabla U\|_{L^2_{\alpha}(H_R^{\boundY})}^2
    +R^{s+1}\tNppar_{H_R^\boundY}(F, f)\bigg)
\\    +\tNpperp_{H_R^{\boundY/2}}(\Dparf^{\betapar}F, \Dparf^{\betapar}f) \bigg],
  \end{multlined}
\end{aligned}
\end{multline*}
where we have used Lemma~\ref{lem:localhighregularity3D-ef} in the last step. 
\end{proof}
\begin{corollary}
Let $\bff\in\cF$ and $\boundY>0$. Let $U$ be the solution of \eqref{eq:minimization}.
Then,
there exists $\gamma > 0$ depending only on $s$, $\Omega$, $\zeta$
and $\boundY$, 
and for every $\varepsilon \in (0,1/2)$, there exists $C_\varepsilon>0$ 
depending additionally on $\varepsilon$
such that for all $\betapar = (\beta_{\parallel,1}, \beta_{\parallel,2})\in \N^2_0$ 
and $ \betaperp\in \N_0$,
with $\ppar = \beta_{\parallel, 1} + \beta_{\parallel, 2}$, $\pperp = \betaperp$, and $p=\ppar+\pperp$,
it holds that
  \begin{equation}
    \|\rf^{\pperp-1/2+\varepsilon} \Dperpf^{\betaperp}\Dparf^{\beta_{\parallel}} \nabla U\|_{L^2_{\alpha}(\omegaf^{\boundY/4})}^2
    \leq
    C_\varepsilon \gamma^{2p+1}p^{2p}   \tNp_{\Omega^\boundY}(F, f).
  \end{equation}
\end{corollary}
\begin{proof}
  This follows directly from Lemma~\ref{lemma:regularity-halfball} with $R\simeq 1$ and from \eqref{eq:CUFf-simplified}.
\end{proof}

\subsubsection {Vertex-face neighborhoods $\omegavf$}
\label{sec:omegavf}
We have 
\begin{align*}
r_{\mathbf{v}} \sim r_{\mathbf{e}} \qquad \text{ and } \qquad r_{\mathbf{f}} \leq r_{\mathbf{e}} \qquad \text{ on $\omegavf$}. 
\end{align*}
\begin{lemma}[Weighted $H^p$-regularity in $\omegavf$] 
\label{lemma:regularity-omegavf}
Let $U$ be the solution of \eqref{eq:minimization}.
There is $\gamma > 0$ depending only on $s$, $\Omega$,
and $\boundY$, 
and for every $\varepsilon \in (0,1/2)$, there exists $C_\varepsilon>0$ 
depending additionally on $\varepsilon$
such that for all $\betapar = (\beta_{\parallel,1}, \beta_{\parallel,2})\in \N^2_0$ and
$ \betaperp\in \N_0$,
with $\ppar = \beta_{\parallel, 1} + \beta_{\parallel, 2}$, $\pperp = \betaperp$, and $p=\ppar+\pperp$, 
it holds that
\begin{equation}
  \label{eq:estimate-omegavf}
\begin{aligned}
  \|\rv^{\ppar+\varepsilon}\rf^{\pperp-1/2+\varepsilon} \Dperpf^{\betaperp}\Dparf^{\beta_{\parallel}} \nabla U\|_{L^2_{\alpha}(\omegavf^{\boundY/4})}^2
  &\leq
C_\varepsilon \gamma^{2p+1}p^{2p}  \tNp_{\Omega^\boundY}(F, f),
\end{aligned}
\end{equation}
where $\Dparf^{\betapar} = \Doneparf^{\beta_{\parallel,
    1}}\Dtwoparf^{\beta_{\parallel, 2}}$.
\end{lemma}
\begin{proof}
  We use the covering of scaled half-balls $H_i = B_{cR_i}(x_i)\cap\Omega$ 
  with $\hH_i = B_{R_i}(x_i)\cap \Omega$ given by Lemma~\ref{lemma:covering-omega_vf}.
  We have, for some constant $C_\boundY > 1$, 
  \begin{equation*}
    \forall i \in \N\quad \forall x \in \hH_i 
    \qquad 
    C^{-1}_\boundY R_i \leq \rv(x)  \leq C_\boundY  R_i. 
  \end{equation*}
  Using
  this and Lemma~\ref{lemma:regularity-halfball}, 
  we obtain
  \begin{multline*}
    \|\rv^{\ppar+\varepsilon}\rf^{\pperp-1/2+\varepsilon} \Dperpf^{\betaperp}\Dparf^{\beta_{\parallel}} \nabla U\|_{L^2_{\alpha}(\omegavf^{\boundY/4})}^2 
    \\
    \begin{aligned}[t]
      &\leq \sum_{i\in \N} (C_\boundY R_i)^{2\ppar+2\varepsilon} \|\rf^{\pperp-1/2+\varepsilon} \Dperpf^{\betaperp}\Dparf^{\beta_{\parallel}} \nabla U\|_{L^2_{\alpha}(H_i^{\boundY/4})}^2
       \\
      &\leq
        \begin{multlined}[t][\arraycolsep]
        \sum_{i\in \N} (C_\boundY R_i)^{2\ppar+2\varepsilon}  \gamma_1^{2p+1}p^{2p} \bigg[R_i^{-2\ppar }
        \bigg(R_i^{-1}\|\nabla U\|_{L^2_\alpha(\hH_i^\boundY)}^2
        + R_i^{s} \tNppar_{\hH_i^\boundY}(F, f)\bigg)
        \\ + \tNpperp_{\hH_i^{\boundY/2}}(\Dparf^{\betapar}F, \Dparf^{\betapar}f)\bigg]
      \end{multlined}
        \\
      &\leq  \gamma^{2p+1}p^{2p}\sum_{i\in \N} \bigg[ \|\rv^{-1/2+\varepsilon}\nabla U\|_{L^2_\alpha(\hH^\boundY_i)}^2 +
        \tNppar_{\hH_i^\boundY}(F, f) + \tNpperp_{\hH_i^{\boundY/2}}(\Dparf^{\betapar}F, \Dparf^{\betapar}f)\bigg].
    \end{aligned}
  \end{multline*}
  The bound $ \rv(x)\geq r_{\partial \Omega}(x)$, the finite overlap of the
  half-balls $\hH_i$, Lemma \ref{lem:estH13D}, and \eqref{eq:CUFf-simplified} conclude the proof.
\end{proof}
\subsubsection{Edge-face neighborhoods $\omegaef$}
\label{sec:omegaef}
We have 
\begin{align*}
r_{\mathbf{f}} \leq r_{\mathbf{e}} \qquad \text{ on $\omegaef$}. 
\end{align*}
We recall the directional coordinates in Def.~\ref{def:gs}.
\begin{lemma}[Weighted $H^p$-regularity in $\omegaef$] 
\label{lemma:regularity-omegaef}
Let $U$ be the solution of \eqref{eq:minimization}.
There is $\gamma > 0$ depending only on $s$, $\Omega$,
and $\boundY$, such that
for every $\varepsilon \in (0,1/2)$, there exists $C_\varepsilon>0$ 
  depending additionally on $\varepsilon$ such that for all $(\betapar, \betaparperp, \betaperp) \in  \N_0^3$,
  $ p = \betapar +  \betaparperp +  \betaperp \in \N_0$,
\begin{equation}
  \label{eq:estimate-omegaef}
  \|\re^{\betaparperp+\varepsilon}\rf^{\betaperp-1/2+\varepsilon} \Dperpg^{\betaperp}\Dparperpg^{\betaparperp}\Dparg^{\betapar} \nabla U\|_{L^2_{\alpha}(\omegaef^{\boundY/8})}^2
  \leq
C_\varepsilon \gamma^{2p+1}p^{2p}  \tNp_{\Omega^\boundY}(F, f)
\;.
\end{equation}
\end{lemma}
\begin{proof}
  We write interchangeably $p_{\bullet}$ and $\beta_\bullet$, for $\bullet \in
  \{\parperp, \parallel, \perp\}$.
  We use the covering of scaled half-balls $H_i = B_{cR_i}(x_i)\cap\Omega$ with 
  $\hH_i = B_{R_i}(x_i)\cap \Omega$ given by Lemma~\ref{lemma:covering-omega_vf}.
  We have, for some constant $C_\boundY > 1$, 
  \begin{equation*}
    \forall i \in \N\quad \forall x \in \hH_i 
    \qquad 
    C^{-1}_\boundY R_i \leq \re(x)  \leq C_\boundY  R_i. 
  \end{equation*}
  Applying Lemma~\ref{lemma:regularity-halfball} to the function
  $\Dparg^{\betapar}U$, which solves \eqref{eq:minimization} with data
  $\Dparg^{\betapar} F$,
  $\Dparg^{\betapar} f$, and remarking that $\gparperp$ is parallel to $\bff$,
  \begin{multline*}
    \|\re^{\pparperp+\varepsilon}\rf^{\pperp-1/2+\varepsilon}\Dperpg^{\betaperp}\Dparperpg^{\betaparperp}\Dparg^{\betapar}  \nabla U\|_{L^2_{\alpha}(\omegaef^{\boundY/8})}^2 
    \\
    \begin{aligned}[t]
      &\leq \sum_{i\in \N} (C_\boundY R_i)^{2\pparperp+2\varepsilon} \|\rf^{\pperp-1/2+\varepsilon} \Dperpg^{\betaperp}\Dparperpg^{\betaparperp}\Dparg^{\betapar} \nabla U\|_{L^2_{\alpha}(H_i^{\boundY/8})}^2
       \\
      &
        \begin{multlined}[t][\arraycolsep]
        \leq \sum_{i\in \N} (C_\boundY R_i)^{2\pparperp+2\varepsilon} \gamma_1^{2(\pperp+\pparperp)+1}(\pperp+\pparperp)^{2(\pperp+\pparperp)} \bigg[R_i^{-2\pparperp }
        \bigg(R_i^{-1}\|\Dparg^{\betapar}\nabla U\|_{L^2_\alpha(\hH_i^{\boundY/2})}^2 +
        \\
        R_i^{s} \tN^{(\pparperp)}_{\hH_i^{\boundY/2}}(\Dparg^{\betapar}F,
        \Dparg^{\betapar}f)\bigg) 
        + \tNpperp_{\hH_i^{\boundY/4}}(\Dparg^{\betapar}\Dparperpg^{\betaparperp}F, \Dparg^{\betapar}\Dparperpg^{\betaparperp}f)\bigg]
      \end{multlined}
        \\
      &
        \begin{multlined}[t][\arraycolsep]
        \leq
        \gamma^{2(\pperp+\pparperp)+1}(\pperp+\pparperp)^{2(\pperp+\pparperp)}\sum_{i\in
          \N} \bigg[ \|\re^{-1/2+\varepsilon}\Dparg^{\betapar}\nabla U\|_{L^2_\alpha(\hH^{\boundY/2}_i)}^2
        +
        \\
        \tN^{(\pparperp+\ppar)}_{\hH_i^{\boundY/2}}(F, f) + \tNpperp_{\hH_i^{\boundY/4}}(\Dparf^{\betapar}F, \Dparf^{\betapar}f)\bigg].
      \end{multlined}
    \end{aligned}
  \end{multline*}
  The bound $ \re(x)\geq r_{\partial \Omega}(x)$ and the finite overlap of the
  half-balls $\hH_i$ imply that we can apply Lemma
  \ref{lem:localhighregularity3D-ef} to obtain, 
  for a constant $C>0$ that
  depends on $\omegaeps$ and on the covering of half-balls, 
  \begin{equation*}
    \sum_{i\in \N} \|\re^{-1/2+\varepsilon}\Dparg^{\betapar}\nabla U\|_{L^2_\alpha(\hH^{\boundY/2}_i)}^2 \leq C R^{-2\ppar-1}
    (\gamma \ppar)^{2\ppar}(1+\gamma \ppar)\bigg(\| \nabla U\|^2_{L^2_\alpha(\tomegaef^\boundY)} 
    +  R^{s+1}\tN^{(\ppar)}_{\tomegaef^\boundY}(F, f)\bigg),
  \end{equation*}
  where $\tomegaef^\boundY$ is a domain that contains the union of the half-balls
  $\hH_i$ and where we can choose $R\simeq 1$.
  Equation \eqref{eq:CUFf-simplified} concludes the proof.
\end{proof}
\subsubsection{Vertex-edge-face neighborhoods $\omegavef$}
\label{sec:omegavef}
We have 
\begin{align*}
r_{\mathbf{f}} \leq r_{\mathbf{e}} \leq r_{\mathbf{v}} \qquad \text{ on $\omegavef$}. 
\end{align*}
We recall the directional coordinates in Def.~\ref{def:gs}.
\begin{lemma}[Weighted $H^p$-regularity in $\omegavef$] 
  \label{lemma:regularity-omegvaef}
  Let $U$ be the solution of \eqref{eq:minimization}.
  There is $\gamma > 0$ depending only on $s$, $\Omega$,
  and $\boundY$, 
  and for every $\varepsilon \in (0,1/2)$, there exists $C_\varepsilon>0$ 
  depending additionally on $\varepsilon$
  such that for all $(\betapar, \betaparperp, \betaperp) \in  \N_0^3$,
  $ p = \betapar +  \betaparperp +  \betaperp \in \N_0$,
  \begin{equation}
    \label{eq:estimate-omegavef}
    \|\rv^{\betapar+\varepsilon}\re^{\betaparperp+\varepsilon}\rf^{\betaperp-1/2+\varepsilon} \Dperpg^{\betaperp}\Dparperpg^{\betaparperp}\Dparg^{\betapar} \nabla U\|_{L^2_{\alpha}(\omegavef^{\boundY/8})}^2
    \leq
    C_\varepsilon \gamma^{2p+1}p^{2p}  \tNp_{\Omega^\boundY}(F, f)
\;.
  \end{equation}
\end{lemma}
\begin{proof}
  We write interchangeably $p_{\bullet}$ and $\beta_\bullet$, 
  for $\bullet \in \{\parperp, \parallel, \perp\}$.
  We use the covering of wedges $W_i$, $\hW_i$ 
  given by Lemma \ref{lemma:covering-omega_vef}.
  We have,
  for some constant $C_W > 1$, 
  \begin{equation*}
    \forall i \in \N\quad \forall x \in \hW_i 
    \qquad 
    C^{-1}_W R_i \leq \rv(x)  \leq C_W  R_i. 
  \end{equation*}
The arguments of Lemma~\ref{lemma:covering-omega_ef} give
a covering $ \bigcup_j H_j \supset W_{i}$
with half-balls $H_j = B_{c_1 \rf(x_j)}(x_j) \cap \Omega$, $x_j \in \bff$ 
and 
scaled half-balls $\hH_j = B_{\hc_1 \rf(x_j)}(x_j)\cap \Omega$ 
such that $\bigcup_j \hH_j \subset \hW_{i}$, 
provided one chooses the parameters $c_1,\hc_1 > 1$ small enough.

Consequently, as in the proof of Lemma \ref{lemma:regularity-omegaef}, we have
  \begin{multline*}
    \|\re^{\pparperp+\varepsilon}\rf^{\pperp-1/2+\varepsilon}\Dperpg^{\betaperp}\Dparperpg^{\betaparperp}\Dparg^{\betapar}
    \nabla U\|_{L^2_{\alpha}(W_i^{\boundY/8})}^2 
    \lesssim
    (\gamma_1 p)^{2p+1}\bigg[R_i^{-2\ppar-1}\bigg(\| \nabla U\|^2_{L^2_\alpha(\hW_i^{\boundY})} \\+  R_i^{s+1}\tN^{(\ppar)}_{\hW_i^{\boundY}}(F, f)\bigg)+\tN^{(\pparperp+\ppar)}_{\hW_i^{\boundY}}(F, f) + \tNpperp_{\hW_i^{\boundY/2}}(\Dparf^{\betapar}F, \Dparf^{\betapar}f)\bigg].
  \end{multline*}
  It follows that
  \begin{multline*}
    \|\rv^{\ppar+\varepsilon}\re^{\pparperp+\varepsilon}\rf^{\pperp-1/2+\varepsilon}\Dperpg^{\betaperp}\Dparperpg^{\betaparperp}\Dparg^{\betapar}  \nabla U\|_{L^2_{\alpha}(\omegavef^{\boundY/8})}^2 
    \\
    \begin{aligned}[t]
      &\leq \sum_{i\in \N} (C_WR_i)^{2\ppar+2\varepsilon} \|\re^{\pparperp+\varepsilon}\rf^{\pperp-1/2+\varepsilon} \Dperpg^{\betaperp}\Dparperpg^{\betaparperp}\Dparg^{\betapar} \nabla U\|_{L^2_{\alpha}(W_i^{\boundY/8})}^2
      \\
      & \lesssim
        \begin{multlined}[t][\arraycolsep]
        (\gamma p)^{2p+1}\sum_{i\in\N}\bigg[\| \rv^{-1/2+\varepsilon}\nabla
        U\|^2_{L^2_\alpha(\hW_i^{\boundY})} \\ +  R_i^{s+\varepsilon}\tN^{\ppar}_{\hW_i^{\boundY}}(F, f)+\tN^{(\pparperp+\ppar)}_{\hW_i^{\boundY}}(F, f) + \tNpperp_{\hW_i^{\boundY/2}}(\Dparf^{\betapar}F, \Dparf^{\betapar}f)\bigg].
      \end{multlined}
    \end{aligned}
  \end{multline*}
The finite overlap of the wedges $\hW_i$, Lemma \ref{lem:estH13D}, 
and equation \eqref{eq:CUFf-simplified} conclude the proof.
\end{proof}
%
\subsubsection{Unified weighted analytic regularity bounds for $U$}
\label{sec:UnifCSBd}
We unify the bounds in all neighborhoods in the following statement.
  \begin{proposition}
\label{prop:CSwgtSum}
    Let $\omega\subset\Omega$ be any set whose boundary intersect at most one
    $\bfv\in \cV$, one $\bfe\in \cE$, and one $\bff\in \cF$.
    Let $(\gperp, \gparperp, \gpar)$ be linearly independent unit vectors as in Def.~\ref{def:gs}.
    Then, there exists $\gamma>0$ such that for all $t<1/2$, there exists
    $C_t>0$ such that
    for all $\beta = (\betaperp, \betaparperp, \betapar)\in\N^3_0$ with $\beta_{\bfe_\perp} = (\beta_\perp, \beta_\parperp)$,
  \begin{equation*}
    \| r_{\partial \Omega}^{-t} r_{\bfv}^{|\beta|} \rho_{\bfv\bfe}^{|\beta_{\bfe_\perp}|} \rho_{\bfe\bff}^{\betaperp} D_{(\gperp, \gparperp, \gpar)}^{\beta} \nabla U \|_{L^2_\alpha(\omega^{\boundY/4})} \leq C_t \gamma^{2\betam+1}\betam^{2\betam}  \tN^{(\betam)}_{\Omega^\boundY}(F, f).
  \end{equation*}
\end{proposition}
\subsection{$H^p$-regularity for the solution $u$ in the polyhedron $\Omega$}
\label{sec:WgtHpRegu}
The preceding analytic regularity bounds on the solution 
$U$ of the CS extension \eqref{eq:extension3D} 
imply corresponding weighted, 
analytic regularity on the weak solution $u$ of the 
integral fractional Laplacian in the polyhedron $\Omega$
ie. \eqref{eq:weakform} via \eqref{eq:extension-b}.
Quantitative control of $u$ in terms of $U$ is achieved via
the multiplicative trace estimate given in the next lemma.
\begin{lemma}
  Let $\boundY>0$. There exists $C_{\mathrm{tr}, \boundY}>0$ such that, for
  all $V:\Omega \times (0, \boundY)\to \R$ with $V(x, \cdot)\in H^1_\alpha((0,
  \boundY))$ for all $x\in \Omega$, it holds that
  \begin{equation}
    \label{eq:trace-estimate}
    \abs{V(x,0)}^2 
    \leq 
    C_{\mathrm{tr}, \boundY} \left( 
      \norm{V(x,\cdot)}_{L^2_\alpha((0, \boundY))}^{1-\alpha}\norm{\partial_y V(x,\cdot)}_{L^2_\alpha((0, \boundY))}^{1+\alpha}
      +
      \|V(x,\cdot)\|^2_{L^2_\alpha((0, \boundY))}
    \right),
  \end{equation}
  where, for a function $v:\Rpos\rightarrow\R$, we write 
  $\|v\|^2_{L^2_\alpha((0, \boundY))}\coloneqq \displaystyle\int_{0}^{\boundY} y^\alpha |v(y)|^2\,dy$.
\end{lemma}
\begin{proof}
 From the proof of \cite[Lem.~3.7]{KarMel19}, we have, for all $W(x, \cdot) \in H^1_{\alpha}(\Rpos)$,
\begin{equation}
\label{eq:L3.7-KarMel19}
 \abs{W(x,0)}^2 
\leq 
C_{\mathrm{tr}} \left( 
\norm{W(x,\cdot)}_{L^2_\alpha(\Rpos)}^{1-\alpha}\norm{\partial_y W(x,\cdot)}_{L^2_\alpha(\Rpos)}^{1+\alpha}
+
\|W(x,\cdot)\|^2_{L^2_\alpha(\Rpos)}
\right).
\end{equation}
Let then $\eta \in C^\infty_0(-\boundY, \boundY)$ with $\eta (0) =  1$ 
and $\|\eta\|_{L^\infty(\R)} +
\|\eta'\|_{L^\infty(\R)}\leq C_\eta$.
Choose $W=\eta V$ in \eqref{eq:L3.7-KarMel19}. We obtain 
\begin{align*}
  \abs{V(x,0)}^2
  &= 
    \abs{(\eta V)(x,0)}^2  
  \\
  &\leq 
    C_{\mathrm{tr}}
    \left( 
    \norm{(\eta V)(x,\cdot)}_{L^2_\alpha(\Rpos)}^{1-\alpha}\norm{(\partial_y (\eta V))(x,\cdot)}_{L^2_\alpha(\Rpos)}^{1+\alpha}
    +
    \|(\eta V)(x,\cdot)\|^2_{L^2_\alpha(\Rpos)}
    \right)
  \\
  & \leq C_{\mathrm{tr}} C_\eta^2 \left( 2\| V(x, \cdot)\|_{L^2_\alpha((0, \boundY))}^{1-\alpha}  \| (\partial_yV)(x, \cdot)\|_{L^2_\alpha((0, \boundY))}^{1+\alpha}  + 3\| V(x, \cdot)\|_{L^2_\alpha((0, \boundY))}^2\right),
\end{align*}
where we have also used that $(a+b)^{1+\alpha}\leq {2}(a^{1+\alpha} + b
^{1+\alpha})$ for all $\alpha\in (-1, 1)$ and all non negative $a, b$.
\end{proof}
\begin{proof}[Proof of Thm.~\ref{thm:mainresult}]
Assume $|\beta|\geq 1$.
Using 
$V= D_{(\gperp, \gparperp, \gpar)}^{\beta} U $ in \eqref{eq:trace-estimate} together with 
multiplication by
$r_{\partial \Omega}^{-2t-2s} r_{\bfv}^{2|\beta|} 
 \rho_{\bfv\bfe}^{2|\beta_{\bfe_\perp}|} \rho_{\bfe\bff}^{2\betaperp}$ 
and integration over $\omega$ leads to
\begin{align*}
  &\norm{r_{\partial \Omega}^{-t-s} r_{\bfv}^{\betapar} r_{\bfe}^{\betaparperp} r_{\bff}^{\betaperp} D_{(\gperp, \gparperp, \gpar)}^{\beta} u}_{L^2(\omega)}^2 
\\
 &\qquad\quad \leq C_{\mathrm{tr},\boundY} 
  \norm{r_{\partial \Omega}^{-t-1} r_{\bfv}^{\betapar} r_{\bfe}^{\betaparperp} r_{\bff}^{\betaperp}  D_{(\gperp, \gparperp, \gpar)}^{\beta} U}_{L^2_\alpha(\omega^{\boundY/4})}^{1-\alpha}
   \norm{r_{\partial \Omega}^{-t} r_{\bfv}^{\betapar} r_{\bfe}^{\betaparperp} r_{\bff}^{\betaperp}   D_{(\gperp, \gparperp, \gpar)}^{\beta}\nabla U}_{L^2_\alpha(\omega^{\boundY/4})}^{1+\alpha} 
\\
  &\qquad\quad\quad + C_{\mathrm{tr},\boundY}   
  \norm{r_{\partial \Omega}^{-t-s} {r_{\bfv}^{\betapar} r_{\bfe}^{\betaparperp} r_{\bff}^{\betaperp}} D_{(\gperp, \gparperp, \gpar)}^{\beta} U}_{L^2_\alpha(\omega^{\boundY/4})}^{2}. 
\end{align*}
On each neighborhood $\omega$, it either holds that
$r_{\partial\Omega} \simeq r_{\bfv}$ 
(when $\partial \omega$ does not intersect with any face or edge of the boundary), 
$r_{\partial\Omega} \simeq r_{\bfe}$  
(when $\partial \omega$ intersects with an edge but no face of the boundary), 
or $r_{\partial\Omega} = r_{\bff}$. 
Consequently,  
as $|\beta|\geq 1$, 
there is a suitable $\widetilde{\beta} \in \N_0^3$ with $|\widetilde{\beta}|=|\beta|-1\geq 0$ 
such that 
\begin{align*}
  \norm{r_{\partial \Omega}^{-t-1} r_{\bfv}^{\betapar} r_{\bfe}^{\betaparperp} r_{\bff}^{\betaperp} 
        D_{(\gperp, \gparperp, \gpar)}^{\beta} U}_{L^2_\alpha(\omega^{\boundY/4})} 
\leq 
  \norm{r_{\partial \Omega}^{-t} r_{\bfv}^{\tbetapar} r_{\bfe}^{\tbetaparperp} r_{\bff}^{\tbetaperp}
        D_{(\gperp, \gparperp, \gpar)}^{\widetilde\beta}\nabla U}_{L^2_\alpha(\omega^{\boundY/4})}.
\end{align*}
Now, the statement follows from Proposition~\ref{prop:CSwgtSum}.

The case $\abs{\beta}=0$ essentially follows from a 1D weighted Hardy inequality similarly as in \cite{FMMS21}. 
Here, we illustrate the argument for the vertex-edge-face case $\omega = \omegavef$, 
noting that the remaining cases correspond verbatim to discussions in \cite{FMMS21}.

We use the coordinates $\{\gpar,\gparperp,\gperp\}$ introduced in
Definition~\ref{def:Coord} and -- by rotation and translation -- assume that the
local orthogonal coordinate system coincides with the canonical coordinates in $\R^3$. 
We introduce the equivalent vertex-edge-face neighborhood
\begin{align*}
\widetilde\omega_{\bf vef}^{\mu,\xi} := \{x \in \Omega: x_1 \in (0,\mu), x_2 \in (0,\xi x_1), x_3 \in (0,\xi x_2)\}
\end{align*}
and drop the superscripts in the following. 
We denote by $\widetilde u$ the function $u$ in the coordinate system in
$\widetilde \omega_{\mathbf{vef}}$.
We remark that there exists $c\geq 1$ such that
in $\widetilde\omega_{\mathbf{vef}}$ holds
\begin{equation}
  \label{eq:rv-xpar-equiv}
  x_{1} \leq r_{\mathbf{v}}(x) \leq c x_{1}, \qquad x_{2} \leq r_{\mathbf{e}}(x) \leq c x_{2}
\end{equation}
and we observe also $r_{\mathbf f}(x) = x_3 = r_{\partial\Omega}(x)$. 
Hence, for almost all $x_1\in (0,\mu)$ and $x_2 \in (0, \xi x_1)$, 
it holds that
\begin{align}
       \label{eq:u-ce-LinfL2}
  \bigg(  x_{3} \mapsto r_{\mathbf{f}}^{1-t-s}(D_{\gperp}\widetilde u)(x) \bigg) \in L^2((0, \xi x_2)) .
\end{align}
Now, the fundamental theorem of calculus, the Cauchy-Schwarz inequality, and 
\eqref{eq:u-ce-LinfL2} imply Hölder continuity of $\widetilde u(x_1,x_2,\cdot)$ for almost all $x_1,x_2$. As
$u \in \widetilde{H}^s(\Omega)$, we can therefore employ the Hardy inequality of 
     \cite[Lem.~{7.1.3}]{Kozlov1997}, which gives
     \begin{align*}
       \|
       r_{\mathbf{f}}^{-t-s}\widetilde u (x_1,x_2,\cdot)
       \|_{L^2((0, \xi x_{2}))}
       \leq C
       \|
       r_{\mathbf{f}}^{1-t-s}(D_{\gperp}\widetilde u)(x_1,x_2,\cdot) 
       \|_{L^2((0, \xi x_{2}))},
     \end{align*}
     with a constant $C$ independent of $x_1,x_2$.
Squaring, integrating in turn over $x_2\in (0,\xi x_1)$ and $x_1 \in (0,\mu)$, and using \eqref{eq:rv-xpar-equiv}, we obtain
     \begin{align*}
       \|
       r_{\partial\Omega}^{-t-s}\widetilde u 
       \|_{L^2(\widetilde\omega_{\mathbf{vef}})} 
       =
       \|
       r_{\mathbf{f}}^{-t-s}\widetilde u 
       \|_{L^2(\widetilde\omega_{\mathbf{vef}})}
       \leq
       C
       \|
       r_{\partial\Omega}^{-t-s} r_{\mathbf{f}} D_{\gperp}\widetilde u \|_{L^2(\widetilde\omega_{\mathbf{vef}})}.
     \end{align*} 
     The term in the right-hand side of the above inequality has been
     bounded in the first part of this proof; 
     this completes the proof except for the fact that 
     the region $\omega_{{\mathbf v}{\mathbf e} \mathbf{f}} \setminus \widetilde{\omega}_{{\mathbf v}{\mathbf e}\mathbf{f}}$ 
     is not covered yet. 
This region can be treated with modifying the parameter $\xi$, exactly as in \cite[Rem.~5.8]{FMMS21}.  
\end{proof}
\section{Conclusion}
\label{sec:Concl} 
For the Dirichlet integral fractional Laplacian $(-\Delta)^s$ 
in a bounded, polytopal domain $\Omega \subset \R^3$,
subject to a source term $f$ which is analytic in $\overline{\Omega}$,
we proved weighted, analytic regularity of weak solutions.
The analysis and the result extends the theory in polygons $\Omega\subset \R^2$, 
developed in our previous work \cite{FMMS21}, to dimension $d=3$.

As is well known from the numerical analysis of Galerkin approximations of solutions for elliptic PDEs,
weighted Sobolev regularity of solutions has direct consequences for the
approximation rate theory of numerical methods: 
boundary weighted Sobolev regularity and Besov regularity has recently been used to 
investigate the convergence rates of first order Galerkin FE discretizations 
on boundary-graded, shape-regular meshes in \cite{BorthNochGrdd23}.
The (boundary- and corner-) weighted analytic regularity proved in \cite{FMMS21} is 
the basis of \emph{exponential convergence rate bounds} for
$hp$-FEM in space dimensions $d=1,2$ \cite{FMMS22_999,FMMS-hp}.

Directions for natural extensions of the present results in three space dimensions suggest themselves:
first, the presently developed proof and the geometric structure of the weights in $\Omega$ 
should facilitate analogous weighted analytic regularity results for integral fractional diffusion such as
$(-\nabla \cdot A(x) \nabla)^s$, 
with an anisotropic diffusion coefficient $A(\cdot)$ being a uniformly
positive definite $d\times d$ matrix, again with analytic in $\overline{\Omega}$ entries.
Likewise, the exponential convergence rate bound established in \cite{FMMS-hp} in the two-dimensional
setting will generalize to the presently considered, polyhedral setting, albeit with 
rate given by $C\exp(-bN^{1/6})$, with $N$ denoting the number of the degrees of freedom of the 
$hp$-FE subspace, and with constants $b,C>0$ depending on $\Omega$, $f$ but not on $N$.
Here, the larger number of geometric situations for $\ge 3$ edges meeting in one, common vertex 
of $\partial \Omega$ will mandate significant extensions and additional technical issues as compared to
the proof in \cite{FMMS-hp}.
Details will be developed elsewhere.

\bibliography{bibliography}
\bibliographystyle{amsalpha}
\appendix
\section{Localization of fractional norms}
\label{sec:AppA}
The following lemma is a slightly improved version of \cite[Lemma A.1]{FMMS21}
\begin{lemma}
\label{lemma:localization-fractional-norms}
Let $R>0$ such that $B_R\subset\Omega$, $c\in (0,1)$, $\eta \in C^\infty_0(B_{cR})$, and $s\in(0,1)$. 
Then,  
\begin{align}
\label{eq:lemma:localization-fractional-norms-10}
\|\eta f\|_{H^{-s}(\Omega)} & \leq C_{\rm loc} \|\eta\|_{L^\infty(B_{cR})} \|f\|_{L^2(B_{cR})},  \\
\label{eq:lemma:localization-fractional-norms-20}
  \|\eta f\|_{H^{1-s}(\Omega)} &
      \begin{multlined}[t][.5\textwidth]
        \leq \Cloctwo \big[\left( R^{s} \|\nabla \eta\|_{L^\infty(B_{cR})}  + (R^{s-1}+1) \|\eta\|_{L^\infty(B_{cR})}\right) \|f\|_{L^2(B_R)}  \\
        + \|\eta\|_{L^\infty(B_{cR})} |f|_{H^{1-s}(B_R)} \big], 
      \end{multlined}
\end{align}
where $C_{\rm loc}$ depends only on $\Omega$ and $s$, and $\Cloctwo$ depends additionally on $c$. 
\end{lemma}
\begin{proof}
\eqref{eq:lemma:localization-fractional-norms-10} 
follows directly from the embedding $L^2\subset H^{-s}$. 
For 
\eqref{eq:lemma:localization-fractional-norms-20}, 
we start from the definition of the Slobodecki semi-norm 
\begin{equation*}
\abs{\eta f}_{H^{1-s}(\Omega)}^2 
= 
\int_{\Omega} \int_{\Omega} \frac{|\eta(x)f(x) - \eta(z)f(z)|^2}{\abs{x-z}^{d+2-2s}}\,dz\,dx
\end{equation*}

We denote the intermediate radius between $R$ and $cR$ as $\tR = \frac{1+c}{2}R$
and write $\tc = \frac{1-c}{2}$ so that $R-\tR = \tR -cR = \tc R$.
We split the integration over $\Omega \times \Omega$ into four subsets, 
\begin{itemize}
\item $B_{\tR} \times B_{R}$, 
\item $B_{\tR}\times B_{R}^c \cap\Omega$, 
\item $B_{\tR}^c \cap\Omega \times B_{cR}$, 
\item $B_{\tR}^c \cap\Omega \times B_{cR}^c \cap\Omega$.
\end{itemize}

For the last case, i.e., for all
$(x, z)\in B_{\tR}^c \cap\Omega \times B_{cR}^c \cap\Omega$, 
we have that $\eta(x)=\eta(z) = 0$ and the integral is zero.
Then, for all
$(x, z) \in B_{\tR}\times B_{R}^c\cap \Omega$, we have $\abs{x-z} \geq \tc R$. 
Hence, using polar coordinates centered at $x$,
\begin{align*}
  &\int_{B_{\tR}} \int_{B_R^c\cap\Omega} \frac{|\eta(x)f(x) - \eta(z)f(z)|^2}{\abs{x-z}^{d+2-2s}}\,dz\,dx = \int_{B_{\tR}} \int_{B_{R}^c\cap\Omega} \frac{|\eta(x)f(x)|^2}{\abs{x-z}^{d+2-2s}}\,dz\,dx \\
  &\qquad\quad\leq \int_{B_{\tR}}\abs{\eta(x)f(x)}^2 \int_{B_R^c}\frac{1}{\abs{x-z}^{d+2-2s}}dz\, dx \lesssim
    \int_{B_{\tR}}\abs{\eta(x)f(x)}^2 \int_{\tc R}^\infty r^{-3+2s}dr\, dx \\
  &\qquad\quad\lesssim (\tc R)^{-2+2s} \norm{\eta}_{L^\infty(B_{cR})}^2 \int_{B_{\tR}}\abs{f(x)}^2  dx \lesssim R^{-2+2s} \norm{\eta}_{L^\infty(B_{cR})}^2 \norm{f}_{L^2(B_{\tR})}^2.
\end{align*}
For the  integration over $B_{\tR}^c\cap \Omega \times B_{cR}$, we write using polar coordinates (centered at $z$)
\begin{align*}
&\int_{B_{\tR}^c\cap \Omega} \int_{B_{cR}} \frac{|\eta(z)f(z)|^2}{\abs{x-z}^{d+2-2s}}\,dz\,dx 
  = 
  \int_{B_{cR}}|\eta(z)f(z)|^2 \int_{B_{\tR}^c\cap\Omega} \frac{1}{\abs{x-z}^{d+2-2s}}\,dx\,dz \\
&\qquad\quad 
  \lesssim \int_{B_{cR}}|\eta(z)f(z)|^2 \int_{\tc R}^\infty \frac{1}{r^{3-2s}}\,dr\,dz 
  \lesssim R^{2s-2}\norm{\eta}_{L^\infty(B_{cR})}^2 \norm{f}_{L^2(B_{cR})}^2.
\end{align*}
Finally, for the integration over $B_{\tR} \times B_{R}$, we use the triangle
inequality
\begin{align*}
  & \int_{B_{\tR}} \int_{B_{R}} \frac{|\eta(x)f(x) - \eta(z)f(z)|^2}{\abs{x-z}^{d+2-2s}}\,dz\,dx
  \\
  &\qquad\quad  \lesssim \int_{B_{\tR}} \int_{B_{R}} \frac{|\eta(x)f(x) - \eta(x)f(z)|^2}{\abs{x-z}^{d+2-2s}} \,dz\,dx 
    + \int_{B_{\tR}} \int_{B_{R}} \frac{|\eta(x)f(z) - \eta(z)f(z)|^2}{\abs{x-z}^{d+2-2s}}\,dz\,dx\\
  &\qquad \quad \eqqcolon (I)+ (II)
\end{align*}
We have
\begin{align*}
  (I) &\leq \| \eta \|_{L^\infty(B_{cR})} \int_{B_{\tR}} \int_{B_{R}} \frac{|f(x) - f(z)|^2}{\abs{x-z}^{d+2-2s}} \,dz\,dx \leq \| \eta \|_{L^\infty(B_{cR})} |f|_{H^{1-s}(B_{R})}.
\end{align*}
Since $\abs{\eta(x)-\eta(z)} 
\leq \norm{\nabla \eta}_{L^\infty(B_{cR})} \abs{x-z}$ and using polar coordinates (centered at $z$) we estimate
\begin{align*}
(II)
&  \leq  \norm{\nabla \eta}_{L^\infty(B_{cR})}^2\int_{B_{R}} |f(z)|^2 \int_{B_{\tR}} \frac{1}{\abs{x-z}^{d-2s}}\,dx\,dz \\
&\lesssim  \norm{\nabla \eta}_{L^\infty(B_{cR})}^2\int_{B_{R}} |f(z)|^2 \int_{0}^{2R} r^{-1+2s}\,dr\,dz \lesssim  \norm{\nabla \eta}_{L^\infty(B_{cR})}^2  \norm{f}_{L^2(B_{R})}^2 R^{2s}.
\end{align*}
The straightforward bound $\|\eta f\|_{L^2(\Omega)}\leq \|\eta\|_{L^\infty(B_{cR})} \| f\|_{L^2(B_{cR})}$
concludes the proof.
\end{proof}
\end{document}